\documentclass[a4paper]{article}
\usepackage{amsmath, amssymb, amsthm, graphicx, 
 bm, comment}
\usepackage[arrow,curve,matrix,arc,2cell]{xy}
\UseAllTwocells

\usepackage{oldgerm}
\usepackage{multicol}

\usepackage[pdftex,colorlinks,urlcolor=blue,pdfstartview=FitH]{hyperref}
\usepackage[left=3.70cm, right=3.70cm, top=3.70cm, bottom=3.70cm]{geometry}

\def\dim{\mathop{\rm dim}}
\def\vdim{\mathop{\rm vdim}}
\def\Corr{\mathop{\rm Corr}}

\def\a{\alpha}

\newcommand{\ZZ}{\mathbb{Z}}
\newcommand{\mR}{\mathbb{R}}

\newcommand{\mU}{\ensuremath{\mathcal{U}}}

\newcommand{\mF}{\ensuremath{\mathcal{F}}}

\newcommand{\mM}{\ensuremath{\mathcal{M}}}
\newcommand{\mP}{\ensuremath{\mathcal{P}}}

\newtheorem{thm}{Theorem}[section]
\newtheorem{lem}[thm]{Lemma}

\newtheorem{cor}[thm]{Corollary}
\newtheorem{prop}[thm]{Proposition}
\newtheorem{rmk}[thm]{Remark}

\theoremstyle{definition}
\newtheorem{defn}[thm]{Definition}

\def\eq#1{{\rm(\ref{#1})}}
\def\cit#1{{\rm\cite{#1}}}

\def\lto{\longrightarrow}

\def\leq{\leqslant}
\def\geq{\geqslant}
\def\Ai{$A_\infty$}

\newcommand{\m}{\mathfrak{m}}
\newcommand{\forg}{\mathfrak{forg}}
\newcommand{\mc}{\mathfrak{c}}

\date{}

\begin{document}

\title{The K\"unneth theorem for the Fukaya algebra of a product of Lagrangians}
\author{Lino Amorim}

\maketitle 

\begin{abstract}
Given a compact Lagrangian submanifold $L$ of a symplectic manifold $(M,\omega)$, Fukaya, Oh, Ohta and Ono construct  a filtered $A_\infty$-algebra $\mathcal{F}(L)$, on the cohomology of $L$, which we call the Fukaya algebra of $L$. In this paper we describe the Fukaya algebra of a product of two Lagrangians submanifolds $L_1\times L_2$. Namely, we show that $\mathcal{F}(L_1\times L_2)$ is quasi-isomorphic to $\mathcal{F}(L_1)\otimes_\infty \mathcal{F}(L_2)$, where $\otimes_\infty$ is the tensor product of filtered \Ai-algebras defined in \cit{Amo2}. As a corollary of this quasi-isomorphism we obtain a description of the bounding cochains on $\mF(L_1\times L_2)$ and of the Floer cohomology of $L_1\times L_2$.
\end{abstract}

\section{Introduction}
Lagrangian Floer cohomology, introduced by Floer in \cite{Flo}, has proven to be a very powerful tool in symplectic topology. This is specially true in the cases of exact or monotone Lagrangians. The general case is much harder because Floer cohomology might not be defined. This is the subject of the work of Fukaya, Oh, Ohta and Ono \cite{FOOO}. For each compact, relatively spin Lagrangian $L$, the authors construct a filtered $A_\infty$-algebra, on the singular chain complex of $L$. There is another version, defined in \cite{Fuk}, that uses the de Rham complex of $L$, which has the advantage of being strictly unital and cyclic. We will refer to this \Ai-algebra as the Fukaya algebra of $L$ and denote it by $\mF(L)$. This is a filtered \Ai-algebra structure on the de Rham complex of $L$, whose operations are defined using a compatible almost complex structure. Loosely speaking, the moduli space of pseudo-holomorphic disks with boundary in $L$ and $k+1$ boundary marked points defines a correspondence between $L^{\times k}$ and $L$. The resulting pull-push map defines the $A_\infty$-map $\m_k$. Then the $\m_0$ term is responsible for the obstructions to defining the Floer cohomology for $L$. More precisely, Floer cohomology can be defined only when there is a deformation of $\mF(L)$ such that $\m_0$ vanishes. Such deformations are given by solutions to the Maurer--Cartan equation on $\mF(L)$, these are called {\it bounding cochains}. A bounding cochain $b$ determines a deformation of the \Ai-maps $\m_k^b$, so that $\m_1^b$ is a differential and so we can define the Floer cohomology of the pair $(L,b)$ as the cohomology of this differential.

In this paper we will study $\mF(L)$ when $L$ is a product Lagrangian. Let $(M_i^{2n_i},\omega_i)$ be a $2n_i$-dimensional symplectic manifold (either compact or convex at infinity) and let $L_i^{n_i}\subseteq M$ be a compact, relatively spin (as defined in \cite[Chapter 8]{FOOO}) Lagrangian submanifold, for $i=1,2$. Then $L_1\times L_2$ is a relatively spin compact Lagrangian submanifold of $(M_1\times M_2,\omega_1\oplus \omega_2)$.
Our main result, Theorem \ref{intmain} below, states that we can describe $\mF(L_1\times L_2)$, in terms of $\mF(L_1)$ and $\mF(L_2)$, using the tensor product of filtered \Ai-algebras introduced in \cite{Amo2}.

\begin{thm}\label{intmain}
 Let $L_1$ and $L_2$ be compact, relatively spin Lagrangian submanifolds of the symplectic manifolds $(M_1,\omega_1)$ and $(M_2,\omega_2)$ respectively. Then we have the following quasi-isomorphism of filtered \Ai-algebras
 $$\mF(L_1 \times L_2) \simeq \mF(L_1) \otimes_\infty \mF(L_2),$$
where $\otimes_\infty$ is the tensor product of filtered \Ai-algebras defined in \cit{Amo2}.
 \end{thm}

 Before going into the proof let us explore some immediate applications of this theorem. First recall that bounding cochains are solutions of the Maurer--Cartan equation,
$$\sum_{k\geq 0}\m_k(b,\ldots, b)=\mP(b) e_L,$$
where $e_L$ is the unit of $\mF(L)$ and $\mP(b)$ is some element in the Novikov ring $\Lambda_0$ (see Section \ref{ainfinity} for the definition). We denote by $MC(L)$ the set of solutions to this equation, modulo gauge equivalence (see \cite[Section 4.3]{FOOO}).

\begin{cor}
 There is a map
 $$\boxtimes: MC(L_1) \times MC(L_2) \lto MC(L_1 \times L_2),$$
 which satisfies $\mP(b_1 \boxtimes b_2) = \mP(b_1)+\mP(b_2)$. Moreover, when $L_1$ and $L_2$ are connected and graded, that is, when their Maslov classes vanish, this map is a bijection.
\end{cor}
\begin{proof}
The first statement follows immediately from Theorem \ref{intmain} and Theorem 1.3 in \cite{Amo2}. For the second statement note that $L_i$ being connected and graded implies that $\mF(L_i)$ is a graded and connected \Ai-algebra (in the sense of Definition 6.7 in \cit{Amo2}). So again, the statement follows from Theorem 1.3 in \cite{Amo2}.
\end{proof}

Given a bounding cochain $b$ we can deform the $A_\infty$-algebra $\mF(L)$ by setting
$$\m_k^b(a_1,\ldots,a_k)=\sum_{i_0,\ldots,i_k}\m_{k+i_0+\ldots+i_k}(b,\ldots,b,a_1,b,\ldots,b,a_k,b,\ldots,b).$$
This defines an \Ai-algebra structure $\mF(L,b)$ on $\Omega^*(L) \hat \otimes \Lambda$, here $\Lambda$ is the Novikov field and $\hat\otimes$ stands for the completed tensor product. It follows from the Maurer--Cartan equation that $\m_1^b$ is a differential. Thus we define the self Floer cohomology $HF(L,b)$ to be the cohomology of $\mF(L,b)$ with respect to the differential $\m_1^b$.

The following corollary is another direct consequence of Theorem 1.3 in \cite{Amo2}
\begin{cor}
 Consider bounding cochains $b_1 \in MC(L_1)$, $b_2 \in MC(L_2)$. Then
 $$\mF(L_1 \times L_2, b_1 \boxtimes b_2) \simeq \mF(L_1, b_1) \otimes_\infty \mF(L_2, b_2).$$
 In particular $$HF(L_1\times L_2, b_1 \boxtimes b_2) \cong HF(L_1,b_1)\otimes_{\Lambda} HF(L_2,b_2).$$
\end{cor}

%This corollary is the first step to understand the Fukaya category of a product of symplectic %manifolds in complete generality. Objects of the Fukaya category

Now we turn to the proof of Theorem \ref{intmain}. The main tool we will use is the following theorem proved in \cite{Amo2}.

\begin{thm}\label{intcriterion}
Let $(A,\m^A)$ and $(B,\m^B)$ be commuting \Ai-subalgebras of $(C,\mu)$ in the sense of Definitions {\rm\ref{subalgebra}} and {\rm\ref{comsubalg}}.
If the map $K:A\otimes B\lto C$ defined as $K(a\otimes b)=(-1)^{\vert a\vert}\mu_{2,0}(a,b)$ is an injective map which induces an isomorphism on $\mu_{1,0}$-cohomology then there is a (strict) quasi-isomorphism
$$A\otimes_\infty B \simeq C.$$ 
\end{thm}

The proof of the next theorem occupies most of this paper. Combined with Theorem \ref{intcriterion} it easily implies Theorem \ref{intmain}. 

\begin{thm}\label{intcomsub}
The Fukaya algebras $\mF(L_1)$ and $\mF(L_2)$ are commuting \Ai-subalgebras of $\mF(L_1\times L_2)$ via the inclusions
  \begin{align}
  \begin{array}{lll}
\mF(L_1) \lto \mF(L_1\times L_2), & & \mF(L_2) \lto \mF(L_1\times L_2)\\
\xi_1 \longmapsto (-1)^{|\xi_1|n_2} p_1^*(\xi_1)                            & &  \xi_2 \longmapsto (-1)^{|\xi_2|n_1} p_2^*(\xi_2),\\
\end{array}\nonumber
 \end{align}
where $p_i: L_1 \times L_2 \lto L_i$ is the projection and $\xi_i \in \Omega^*(L_i)$.
 \end{thm}
 
The quasi-isomorphism class of the Fukaya algebra $\mF(L)$ is an invariant of the Lagrangian submanifold, but its construction depends on the choice of a compatible almost complex structure and choices of perturbations of the moduli spaces of stable disks with boundary in $L$. The theorem above should be interpreted as saying that after fixing almost complex structures and perturbations of the relevant moduli spaces for $L_1$ and $L_2$, if we take the product almost complex structure on $M_1\times M_2$, there are specific choices of perturbations of the moduli spaces of disks with boundary in $L_1\times L_2$ so that $\mF(L_1)$ and $\mF(L_2)$ are commuting subalgebras of $\mF(L_1\times L_2)$.
 
\vspace{.1cm}

\begin{proof}[Proof of Theorem \ref{intmain}]In view of Theorems \ref{intcriterion} and \ref{intcomsub} we need only to check that $K$ is an injective map which induces an isomorphism in $\mu_{1,0}$-cohomology. Recall (or see Section 2) that $\m_{1,0}^{L}=(-1)^{\dim L +1}d$ and $\mu_{2,0}(a, b)= (-1)^{|a|}a \wedge b$. Then, by definition
$$K(\xi_1 \otimes \xi_2)=(-1)^{|\xi_1|(1+n_2)+|\xi_2|n_1} p_1^*(\xi_1) \wedge p_2^*(\xi_2).$$
Thus, up to a change in sign, this is simply the usual K\"unneth map. This is clearly injective and the standard K\"unneth Theorem on de Rham cohomology implies the claim.
\end{proof}

We finish the introduction with an illustration of the ideas involved in the proof of Theorem \ref{intcomsub}. We sketch the proof of one of the identities that are part of the definition of commuting \Ai-subalgebras, namely
$$\m_{1,\beta}(\xi)=0, \ \beta=\beta_1\times \beta_2 \in \pi_2(M_1\times M_2, L_1 \times L_2), \text{when} \ \beta_1,\beta_2 \neq 0.$$
We consider the moduli space of stable disks with boundary in $L_1\times L_2$ and two boundary marked $\mM_{2}(\beta)$. Evaluation at the two marked points gives maps $ev_0, ev_1 : \mM_{2}(\beta) \lto L_1\times L_2$. Roughly speaking, that is pretending that $\mM_{2}(\beta)$ is a smooth manifold and that $ev_0$ is a submersion, we define 
$$\m_{1,\beta}(\xi)= (ev_0)_*(ev_1^*(\xi)),$$
where $(ev_0)_*$ is fiber integration.
In our situation, there is a map 
$$\pi: \mM_{2}(\beta) \lto \mM_{2}(\beta_1)\times \mM_{2}(\beta_2),$$ 
which sends a stable map $u$ to its components $(p_1 \circ u, p_2\circ u)$ and stabilizes the domains if necessary. Observe that the evaluation maps factor through $\pi$ and a simple computation shows that the target of $\pi$ has smaller dimension than the domain. Then the claim follows from the general fact about fiber integration: given maps $f,g,h$ such that $h=g\circ f$, with $f: M \lto N$ and $\dim N < \dim M$ we have $h_*(f^* (\xi))=0$. 

To actually prove the claim we have to carry out a similar argument in the context of spaces with Kuranishi structures. We would like to point out that a similar argument is used by Fukaya in \cite{Fuk} to show that $\mF(L)$ is a strictly unital \Ai-algebra.

This paper is organized in the following way. In Section \ref{ainfinity}, we give some background on filtered \Ai-algebras. In Section \ref{secfuk}, we review the construction of the Fukaya algebra following \cite{Fuk}, but describing several sign conventions that were not explicit. In Section \ref{secfukprod} we prove the modulo $T^E$ version of Theorem \ref{intcomsub}, assuming the existence of some particular Kuranishi structures on the moduli spaces of disks and in \S \ref{proofs} we construct these Kuranishi structures. In Section \ref{secfukprodmodte}, we upgrade the result of \S \ref{secfukprod} from \Ai-algebras modulo $T^E$ to full-fledged \Ai-algebras, thus completing the proof of Theorem \ref{intcomsub}. We finish with an appendix where we define fiber integration and smooth correspondences and prove several useful properties these satisfy. 
\vspace{.3cm}

\noindent{\bf Acknowledgements:} This paper is a reinterpretation of some of the results in my Ph.D. thesis. I would like to thank my advisor Yong-Geun Oh for his continued help and support. I would also like to thank Kenji Fukaya, Dominic Joyce, Hiroshi Ohta and Kaoru Ono for helpful conversations.  I would also like to thank an anonymous referee for pointing out a mistake in a previous version of the paper.
During my Ph.D. I was partially supported by FCT through the scholarship  SFRH/ BD/30381/2006. During the preparation of this paper I was supported by EPSRC grant EP/J016950/1.

\vspace{.3cm}

\noindent{\bf Conventions:} Given an element $a$ in a graded vector space $A$, we will denote its degree by $\vert a \vert$. We will also use a shifted degree, $\vert\vert a \vert\vert = \vert a \vert -1$.

Let $M$ and $N$ be smooth manifolds and let $p_1:M\times N \lto M$ and $p_2:M\times N \lto N$ be the natural projections. Given differential forms $\xi_1 \in \Omega^*(M)$ and $\xi_2 \in \Omega^*(N)$, we denote by $\xi_1 \times \xi_2$ the differential form $p_1^*(\xi_1)\wedge p_2^*(\xi_2) \in \Omega^*(M\times N)$.

\section{Filtered $A_\infty$-algebras}\label{ainfinity}

In this section we briefly review some basic notions of the theory of \Ai-algebras. For more complete treatments we refer the reader to \cite{FOOO} for the case of filtered \Ai-algebras and to \cite{Seisub} for the classical case. We will also review the notions of commuting \Ai-subalgebras introduced in \cite{Amo2}.

\begin{defn}
 An \Ai-algebra over a ring $R$ consists of a $\ZZ_2$-graded $R$-module $A$ and a collection of multilinear maps $\m_{k}:A^{\otimes k} \lto A$ for each $k \geq 0$ of degree $k \pmod 2$ satisfying the following equation
\begin{align} \label{Ainf}
 \sum_{\substack{0\leq j\leq n\\1\leq i\leq n-j+1}}(-1)^{*}\m_{n-j+1}(a_1,\ldots,\m_{j}(a_{i},\ldots,a_{i+j-1}),\ldots,a_n)=0
\end{align}
where $*=\sum_{l=1}^{i-1} \vert\vert a_l \vert\vert$.
\end{defn}

We are interested in a particular kind of \Ai-algebra defined over the Novikov ring 
$$\Lambda_0=\Big\{\sum_{i=0}^\infty a_iT^{\lambda_i}\vert \lambda_i, a_i\in \mR, 0\leq\ldots\leq \lambda_i\leq\lambda_{i+1}\leq\ldots,\lim_{\lambda_i\to\infty}= +\infty\Big\}.$$
Note that $\Lambda_0$ has a natural filtration 
$$F^E\Lambda_0=\Big\{\sum_i a_iT^{\lambda_i}\vert \lambda_i\geq E,\forall\ i\textrm{ with } a_i\neq0 \Big\}.$$
Moreover $\Lambda_0$ is local and localizing at the maximal ideal we obtain the Novikov field
$$\Lambda=\Big\{\sum_{i=0}^\infty a_iT^{\lambda_i}\vert \lambda_i, a_i\in\mR, \lambda_i\leq\lambda_{i+1},\lim_{\lambda_i\to\infty}= +\infty\Big\}.$$

Next consider $G\subset \mathbb{R}_{\geq 0}\times2\mathbb{Z}$ and and denote by $E:G\lto \mathbb{R}_{\geq 0}$ and $\mu:G\lto 2\mathbb{Z}$ the natural projections. We say $G$ is a {\it discrete submonoid}, if it is an additive submonoid satisfying
$$E^{-1}([0,c])\textrm{ is finite for any }c\geq 0.$$

\begin{defn}
 Let $G$ be a discrete submonoid, a $G$-gapped filtered \Ai-algebra $A=(A, \m)$ consists of a $\ZZ$-graded real vector space $A$ together with maps $\m_{k,\beta}:A^{\otimes k} \lto A$, for each $\beta \in G$ and $k\geq0$ of degree $2-k-\mu(\beta)$. These are required to satisfy $\m_{0,0}=0$ and for all $\beta \in G$ and homogeneous $a_1,\ldots,a_n \in A$:
\begin{align}\label{gapAinf}
\sum_{\substack{\beta_1+\beta_2=\beta\\0\leq j\leq n\\1\leq i\leq n-j+1}}(-1)^{*}\m_{n-j+1,\beta_2}(a_1,\ldots,\m_{j,\beta_1}(a_{i},\ldots,a_{i+j-1}),\ldots,a_n)=0.
\end{align}

Fix $E>0$, if the maps $\m_{k,\beta}$ only exist for $\beta$ such that $E(\beta)\leq E$ and the above condition is satisfied for all such $\beta$ we say $(A,\m)$ is an \Ai-algebra modulo $T^E$.
\end{defn}

We say $A$ is a filtered \Ai-algebra if it is a $G$-gapped filtered \Ai-algebra for some discrete submonoid $G$.

The reason for the name \Ai-algebra modulo $T^E$ is as follows. Given a filtered \Ai-algebra $(A,\m)$, let $\hat{A}_0= A\hat\otimes\Lambda_0$ be the completion of $A\otimes_{\mR}\Lambda_0$ with respect to filtration induced by filtration in $\Lambda_0$. Then define maps $\m_k:\hat{A}_0^{\otimes k} \lto \hat{A}_0$ by setting
$$  \m_k=\sum_{\beta\in G} \m_{k,\beta}T^{E(\beta)}.$$
The gapped condition ensures this well defined and (\ref{gapAinf}) implies that $(\hat{A}_0, \m)$ is an \Ai-algebra over $\Lambda_0$. Similarly, given an \Ai-algebra modulo $T^E$ we can construct an \Ai-algebra $\hat{A}_0= A\hat\otimes(\Lambda_0/F^E\Lambda_0)$.

\begin{defn}
Let $(A,\m)$ be a filtered \Ai-algebra. If $G \subset \mR_{\geq 0} \times \{0\}$, we say $A$ is graded.
 
The \Ai-algebras is said to be unital if there is an element $e_A\in A$ of degree $0$ (called the unit) satisfying
$$\m_{2,0}(e_A,a)=(-1)^{\vert a\vert}\m_{2,0}(a,e_A)=a$$ and $\m_{k,\beta}(\ldots,e,\ldots)=0$ for $(k,\beta)\neq (2,0)$.
\end{defn}

We end this section by recalling the definitions of subalgebra and commuting subalgebras from \cite{Amo2}. We also give the modulo $T^E$ version of these definitions.

\begin{defn}\label{subalgebra}
Let $(A,\m^A)$ and $(C,\mu)$ be (respectively) $G_A$ and $G$-gapped filtered \Ai-algebras, for discrete submonoids $G_A\subseteq G$. We say A is a subalgebra of $C$ if $A\subseteq C$, $e_A=e_C$ and for all $k>0$ and $a_1, \ldots a_k \in A$ we have
\begin{align}
 \mu_{k,\beta}(a_1,\ldots,a_k)&=\m_{k,\beta}^A(a_1,\ldots,a_k), \ \  \beta\in G_A,\nonumber \\
 \mu_{k,\beta}(a_1,\ldots,a_k)&=0, \  \ \beta\in G\setminus G_A.\nonumber
\end{align}

If $(A,\m^A)$ and $(C,\mu)$ are \Ai-algebras modulo $T^E$, we say $A$ is a subalgebra modulo $T^E$ of $C$ if the above conditions hold for all $\beta$ satisfying $E(\beta)\leq E$.
\end{defn} 

\begin{defn}\label{comsubalg}
Let $(A,\m^A)$ and $(B,\m^B)$ be $G_A$ and $G_B$-gapped filtered \Ai-algebras. Suppose $A$ and $B$ are subalgebras of $(C,\mu)$ a $G$-gapped \Ai-algebra with $G=G_A+G_B$. Denote by $K:A\otimes B\lto C$ the map defined as $K(a\otimes b)=(-1)^{\vert a\vert}\mu_{2,0}(a,b)$. We say $A$ and $B$ are \emph{commuting subalgebras} if given $c=K(a\otimes b)$ and $c_1,\ldots, c_k\in C$ such that for each $i$, $c_i=a_i$ or $c_i=b_i$ for some $a_i\in A$ and $b_i\in B$, the following conditions hold.
\begin{itemize}
\setlength{\parsep}{0pt}
\setlength{\itemsep}{0pt}
\item[{\bf(a)}] For $k>0$, $\mu_{k,\beta}(c_1,\ldots, c_k)=0$ unless
\begin{itemize}
\setlength{\parsep}{0pt}
\setlength{\itemsep}{0pt}
\item[{\bf(i)}] $(k,\beta)=(2,0)$ and $c_1 \in A$, $c_2 \in B$ (or vice-versa) in which case, $\mu_{2,0}(c_1,c_2)+(-1)^{\|c_1\|\|c_2\|}\mu_{2,0}(c_2,c_1)=0$,
\item[{\bf(ii)}] $c_i=a_i$ for all $i$ and $\beta\in G_1$,
\item[{\bf(iii)}] $c_i=b_i$ for all $i$ and $\beta\in G_2$.
\end{itemize}
\item[{\bf(b)}] $\mu_{0,\beta}=\m^A_{0,\beta}+\m^B_{0,\beta}$, with the convention that $\m^A_{0,\beta}=0$ (respectively $\m^B_{0,\beta}$) if $\beta\notin G_A$ (respectively $\beta\notin G_B$).
\item[{\bf(c)}] $\mu_{k+1,\beta}(c_1,\ldots,c_i,c,c_{i+1},\ldots,c_k)=0$ unless
\begin{itemize}
\setlength{\parsep}{0pt}
\setlength{\itemsep}{0pt}
\item[{\bf(i)}] $c_i=a_i$  for all $i$ and $\beta\in G_A$, in which case it equals
$$(-1)^{|b|\sum_{j> i}\|a_j\|}K(\m^A_{k+1,\beta}(a_1\ldots a_i,a,\ldots,a_k)\otimes b),$$
\item[{\bf(ii)}] $c_i=b_i$ for all $i$ and $\beta\in G_B$, in which case it equals
$$(-1)^{|a|\left(\sum_{j\leq i}\|b_j\|+1\right)}K(a\otimes \m^B_{k+1,\beta}(b_1\ldots b_i,b,\ldots,b_k)).$$
\end{itemize}
\end{itemize}

If $(A,\m^A)$, $(B, \m^B)$ and $(C,\mu)$ are \Ai-algebras modulo $T^E$, we say $A$ and $B$ are commuting subalgebras modulo $T^E$ of $C$ if the above conditions hold for all $\beta$ satisfying $E(\beta)\leq E$.
\end{defn}

\section{Fukaya algebra}\label{secfuk}

In this section we will review the construction of the \emph{Fukaya algebra} $\mF(L)$ of a relatively spin, compact Lagrangian $L$. This was constructed by Fukaya in \cite{Fuk} building on the work of Fukaya, Oh, Ohta and Ono in \cite{FOOO}, \cite{FOOOt1} and \cite{FOOOt2}. We refer the reader to \cite{Fuk} for a complete discussion of this construction. Throughout the section we will assume the reader is familiar with spaces with Kuranishi structures and good coordinate systems as defined in \cite[Appendix A1]{FOOO}. We refer the reader to \cite{FOOOtech} for a detailed exposition of these concepts.

In \S \ref{flmodte} we use the moduli spaces of stable disks to construct an \Ai-algebra modulo $T^E$ on the de Rham complex of $L$ for any $E\in \mathbb{R}_{>0}$. In \S \ref{modte} we will review how one can take the limit of this construction to obtain an \Ai-algebra. 

\subsection{$\mathcal{F}(L)$ modulo $T^E$}\label{flmodte}

Let $(M^{2n},\omega)$ be a $2n$-dimensional symplectic manifold, where $M$ is either compact or convex at infinity. Consider $L$ a compact Lagrangian submanifold that is either spin, or more generally \emph{relatively spin}, that is, $L$ is oriented and there is a degree two cohomology class $\sigma\in H^2(M,\mathbb{Z}_2)$ that restricts to the second Stiefel-Whitney class of $L$, that is $\sigma\vert_{L}=w_2(TL)$. We actually will need to fix a relative spin structure on $L$. 

The vector space underlying the Fukaya algebra $\mF(L)$ is $\mF(L):=\Omega^*(L)$, the de Rham complex of $L$. 
In order to construct the $A_\infty$-operations, we fix an almost complex structure $J$ on $M$ compatible with $\omega$. Now given $\beta\in \pi_2(M,L)$ consider $\mM_{k+1}(\beta)$ the (compactified) moduli space of $J$-holomorphic disks with $k+1$ boundary points,  with boundary on $L$ and homotopy class $\beta$. Denote by $[(\Sigma,\vec{z}),u]$ an element of $\mM_{k+1}(\beta)$ where $\Sigma$ is a semi-stable disk, $\vec{z}=(z_0,\ldots,z_k)$ are $k+1$ marked points on $\partial\Sigma$ respecting the cyclic order and $u:(\Sigma,\partial\Sigma)\lto(M,L)$ is a $J$-holomorphic map such that $[u]=\beta$.
There are natural evaluation maps:
\begin{align}
ev_i:&\mM_{k+1}(\beta)\lto L\nonumber\\
&((\Sigma,z_0,\ldots,z_k),u)\longmapsto u(z_i).\nonumber
\end{align}
It is proven in \cite{FOOO} that $\mM_{k+1}(\beta)$ has a Kuranishi structure (see \cite[Appendix A1]{FOOO} for the definition) with corners of virtual dimension $\vdim=n+\mu(\beta)+k-2$. Additionally we have the following description of its boundary:
\begin{align}\label{boundary}
\partial \mM_{k+1}(\beta)=\bigcup_{\substack{\beta_1+\beta_2=\beta\\0\leq j\leq k\\1\leq i\leq k-j+2}}(-1)^{n+i(1+j)}\mM_{j+1}(\beta_1)\ {}_{ev_0}\times_{ev_i}\mM_{k-j+2}(\beta_2),
\end{align} where the equality holds as spaces with oriented Kuranishi structures. In this formulation, the statement about the orientations can be found in \cite[Theorem 5.9]{AkaJoy}.

In \cite{Fuk}, Fukaya showed that this construction can be carried out in a way compatible with forgetting (boundary) marked points. More precisely, for each $k\geq0$ and $0\leq i \leq k+1$ consider the map
$$\forg_i:\mM_{k+1}(\beta)\lto\mM_k(\beta)$$
that forgets the $i$-th marked point, and collapses any irreducible components that become unstable. We require that the Kuranishi structures are compatible with the maps $\forg_i$ in the following sense.

\begin{defn}\label{kurcompatible}
Let $\varphi: X\lto Y$ be a continuous map between spaces with Kuranishi structures. We say the Kuranishi structures \emph{compatible} (with respect to $\varphi$) if for every $p\in X$ and $q=\varphi(p)$, there is a map between the Kuranishi neighborhoods $(V_p,E_p,\Gamma_p,s_p,\psi_p)$ and $(V_q,E_q,\Gamma_q,s_q,\psi_q)$. The map consists of the following data:
\begin{itemize}
\setlength{\parsep}{0pt}
\setlength{\itemsep}{0pt}
\item[{\bf(a)}] an injective homomorphism $h_{pq}:\Gamma_p\lto\Gamma_q$;
\item[{\bf(b)}] a continuous, $h_{pq}$-equivariant map $\varphi_{pq}:V_p\lto V_q$, that is smooth on a dense, open subset of $V_p$;
\item[{\bf(c)}] an isomorphism $E_p\simeq \varphi_{pq}^*E_q$;
\item[{\bf(d)}] $s_p=\varphi_{pq}^*s_q$;
\item[{\bf(e)}] $\varphi\circ\psi_p=\psi_q\circ\varphi_{pq}$ on $s^{-1}_p(0)/\Gamma_p$.
\end{itemize}
\end{defn}

This definition is a slight weakening of the definition given in Sections 3 of \cite{Fuk}. Fukaya in \cite[Corollary 3.1]{Fuk} shows the following:
\begin{prop}\label{compKur}
There exist Kuranishi structures in $\mM_{k+1}(\beta)$ that are compatible, in the sense of Definition {\rm\ref{kurcompatible}}, with $\forg_i$, the map
$$ev_0:\mM_{k+1}(\beta)\lto L$$ 
is weakly submersive and the decomposition of the boundary (\ref{boundary}) holds as spaces with Kuranishi structures.
\end{prop}

Moreover, Fukaya showed that this Kuranishi structures admit good coordinate systems and systems of transversal multisections compatible with the forgetful map, in the following sense.

\begin{defn}
Let $\varphi:X\lto Y$ be a continuous map between spaces $X$ and $Y$ with compatible Kuranishi structures. Good coordinate systems $\{\mathcal{U}_\alpha\}_{\alpha\in I}$ on $X$ and $\{\mathcal{V}_\beta\}_{\beta\in J}$ are said to be \emph{compatible} if there is an order-preserving map $I\lto J$, $\alpha\lto \beta (\alpha)$ and a map between the Kuranishi neighborhoods $\mathcal{U}_\alpha$ and $\mathcal{V}_{\beta(\alpha)}$ .   
\end{defn}

\begin{defn}\label{corrcompatible}
Let $\{\mathcal{U}_\alpha\}_{\alpha\in I}$ and $\{\mathcal{V}_\beta\}_{\beta\in J}$ be compatible good coordinate systems on $X$ and $Y$. Continuous families of multisections $(W_\alpha, S_\alpha)_{\alpha\in I}$ and  $(W_\beta, S_\beta)_{\beta\in J}$ are said to be \emph{compatible} if $W_\alpha=W_{\beta(\alpha)}$, $\theta_\alpha=\theta_{\beta(\alpha)}$ and $S_\alpha=S_\beta\circ(\textrm{id}\times\varphi_{\alpha\beta})$. Moreover we require that $\textrm{id}\times\varphi_{\alpha\beta} |_{S_{\alpha}^{-1}(0)}$ is smooth in a dense, open subset of $S_{\alpha}^{-1}(0)$.
\end{defn} 

Again these are small modifications of the definitions given in Section 5 of \cite{Fuk}.

\begin{prop}[Fukaya \cit{Fuk}]\label{compmulti}
For each $\epsilon$ and $E> 0$, there exist continuous families of transversal multisections on $\mM_{k+1}(\beta)$, for $k\geq 0$ and $\omega\cap\beta\leq E$, which are $\epsilon$-small. These are compatible, in the sense of Definition \ref{corrcompatible}, with $\forg_i$ and the evaluation maps $(ev_0)_\a\vert_{S^{-1}_\a(0)}$ are submersive. Moreover given the decomposition of the boundary \eq{boundary}, the restriction of the multisections to the boundary agrees with the fiber product of multisections on the right-hand side of \eq{boundary}.
We will denote this by a system of compatible multisections.
\end{prop}

We are now ready to define the \Ai-operations. First define 
$$NE(L)=\{\beta\in \pi_2(M,L,\mathbb{Z})\vert \mM_1(\beta)\neq\emptyset\}.$$
Since $L$ is orientable the Maslov index $\mu(\beta)$ is always an even number. We consider the map
$$E\oplus\mu:NE(L)\lto\mathbb{R}_{\geq0}\times2\mathbb{Z},$$
where $E(\beta)=\omega \cap \beta$ and $\mu$ is the Maslov class. Denote  by $G(L)$ the submonoid generated by its image. Gromov's compactness implies that the number of elements $\beta\in NE(L)$ such that $E(\beta)\leq E$ for fixed $E$ is finite. Therefore $G(L)$ is a discrete submonoid.

Fix $E> 0$ and a system of multisections provided by Proposition \ref{compmulti}. For each $\beta \in NE(L)$, $k\geq 0$, such that $E(\beta)\leq E$ and $(k,\beta)\neq(1,0)$, given $\xi_1,\ldots,\xi_k\in\Omega^*(L)$, we define:
$$\m'_{k,\beta}(\xi_1,\ldots,\xi_k)= (-1)^{\sum_{i=1}^k(k-i)\vert\xi_i\vert}\Corr\big(ev_1\times\ldots\times ev_k,\mM_{k+1}(\beta)^S,ev_0\big)(\xi_1\times\ldots\times\xi_k),$$
where $\Corr$ is the smooth correspondence map defined in \cite{FOOOt2}, that we review in \S \ref{smoothcorr}. For $\beta \in G(L)$ we define
$$\m'_{k,\beta}= \sum_{\substack{\beta'\\ E\oplus\mu(\beta')=\beta}} \m'_{k,\beta'}.$$
In the remaining case $(k,\beta)=(1,0)$ we set $\m'_{1,0}(\xi)=(-1)^{n+1}d\xi$, where $d$ is the de Rham differential.
Finally we set 
$$\m_{k,\beta}=(-1)^{\frac{(k-1)(k-2)}{2}}\m'_{k,\beta}.$$

\begin{prop}\label{aalgebra} The maps $\m_{k,\beta}$ define a filtered $A_\infty$-algebra modulo $T^E$ on the de Rham complex $\Omega^*(L)$. That is for each $\beta$ such that $E(\beta)\leq E$ and $k\geq 0$ we have
\begin{align}\label{Ainfeq}
\sum_{\substack{\beta_1+\beta_2=\beta \\0\leq j\leq k\\1\leq i\leq k-j+1}}(-1)^{\sum_{l=1}^{i-1} ||\xi_l||}\m_{k-j+1,\beta_1}(\xi_1,\ldots,\m_{j,\beta_2}(\xi_{i},\ldots,\xi_{i+j-1}),\ldots,\xi_k)=0.
\end{align}
\end{prop}
\begin{proof} Proposition \ref{kurstokes} implies that
\begin{align}
&(-1)^{k+1}d \Corr\left(ev_1\times\ldots\times ev_k,\mM_{k+1}(\beta)^S,ev_0\right)(\xi_1\times\ldots\times\xi_k)\nonumber\\
&+\sum_{i=1}^{k}(-1)^{\sum_{l=1}^{i-1}|\xi_l|}\Corr\left( ev_1\times\ldots\times ev_k,\mM_{k+1}(\beta)^S,ev_0\right)(\xi_1\times\ldots\times d\xi_i\times\ldots\xi_k)\nonumber\\
&=\Corr\left(ev_1\times\ldots\times ev_k,\partial\mM_{k+1}(\beta)^S,ev_0\right)(\xi_1\times\ldots\times\xi_k)\nonumber\\
&=\sum_{\substack{\beta_1+\beta_2=\beta\\0\leq j\leq k}}(-1)^{\epsilon}\Corr\left(ev_1\times\ldots\times ev_k,\mM_{j+1}(\beta_2)^S\ _{ev_0}\times_{ev_i}\mM_{k-j+2}(\beta_1)^S, ev_0\right)\big( \ \vec\xi \ \big),\nonumber
\end{align}
where $\epsilon=n+i(j+1)$ and the last equality follows from Proposition \ref{compmulti}. In turn, Proposition \ref{composition} implies
\begin{align}
&\Corr\left(ev_1\times\ldots\times ev_k,\mM_{j+1}(\beta_2)^S\ _{ev_0}\times_{ev_i}\mM_{k-j+2}(\beta_1)^S,ev_0\right)(\xi_1\times\ldots\times\xi_k)=\nonumber\\
&=(-1)^{(\sum_{l=1}^{i-1}|\xi_l|)j}\Corr\left(ev_1\times\ldots\times ev_{k-j+2},\mM_{k-j+2}(\beta_1)^S, ev_0\right)\left(\xi_1\times\ldots\times\xi_{i-1}\times\right.\nonumber\\
&\hspace{2cm}\left.\Corr(ev_1\times\ldots\times ev_j,\mM_{j+1}(\beta_2)^S,,ev_0)(\xi_i\times\ldots\times \xi_{i+j-1})\times\ldots\times\xi_k\right).\nonumber
\end{align}
Introducing the signs in the definition of $\m_{k,\beta}$, a straightforward computation shows the $A_\infty$-equation.
\end{proof}

Next we want to show that $(\Omega^*(L), \m_{k,\beta})$ has a unit. For these we need to use compatibility with the  forgetful maps together with the following proposition whose proof we postpone to \S \ref{smoothcorr}.

\begin{prop}\label{kurvanishingpi*}
Let $X$ and $Y$ be Kuranishi spaces. Let $f':X\lto M$ and $g':Y\lto M$ be smooth strongly continuous maps and $f: X\lto N$, $g:Y\lto N$ be weak submersions. Also, let $(W_\alpha, S_\alpha)_{\alpha\in I}$ and  $(W_\beta, S_\beta)_{\beta\in J}$ be compatible continuous families of multisections such that $f_\alpha\vert_{S^{-1}_\alpha(0)}$ and $g_\beta\vert_{S^{-1}_\beta(0)}$ are submersions. Assume that $f=g\circ\varphi$, $f'=g'\circ\varphi$ and $\vdim X>\vdim Y$. Then $$\Corr(f',X^S,f)(\xi)=0$$ for any $\xi\in\Omega^*(M)$.
\end{prop}

Assuming this result we can easily prove

\begin{prop}
 Let $e_L \in \Omega^0(L)$ be the constant function equal to one. Then $e_L$ is a unit for \Ai-algebra modulo $T^E$ $(\Omega^*(L), \m_{k,\beta})$, that is
$$\m_{k,\beta}(\xi_1,\ldots,\xi_{i-1},e_L,\xi_{i+1},\ldots,\xi_k)=0$$
for all $(k,\beta)\neq(2,0)$ and $\m_{2,0}(e_L,\xi)=(-1)^{\vert\xi\vert}\m_{2,0}(\xi,e_L)=\xi$.
\end{prop}
\begin{proof}
 Consider the forgetful map $\forg_i:\mM_{k+1}(\beta)\lto\mM_k(\beta)$. The system of multisections is compatible with this map and the evaluation maps $ev_j$ for $j\neq i$ factor through this map. Therefore Proposition \ref{kurvanishingpi*} immediately implies $\m_{k,\beta}(\xi_1,\ldots,\xi_{i-1},e_L,\xi_{i+1},\ldots,\xi_k)=0$ for all $(k,\beta)\neq(2,0)$. The last statement is obvious.
\end{proof}

\subsection{From $A_\infty$-algebra modulo $T^E$ to full-fledged $A_\infty$-algebra}\label{modte}

So far we have constructed an $A_\infty$-algebra on $\mF(L)$ modulo $T^E$ for arbitrary $E>0$. We will now explain how to obtain an actual $A_\infty$-algebra. In fact, this argument can be easily adapted to prove that $\mF(L)$ is independent of the almost complex structure and the choices of systems of multisections up to quasi-isomorphism. 

We start with the definition of pseudoisotopy between \Ai-algebras given in \cite[Section 9]{Fuk}. We will restrict ourselves to \Ai-algebras on $\Omega^*(L)$ and we will ignore the cyclic structures.
\begin{defn}\label{defisotopy}
Let $\m^{(0)}$ and $\m^{(1)}$ be filtered $G$-gapped \Ai-algebra structures on $\Omega^*(L)$. A pseudoisotopy between $\m^{(0)}$ and $\m^{(1)}$ is a pair $(\m^t_{k,\beta}, \mc^t_{k,\beta})$, where 
$$\m^t_{k,\beta}, \mc^t_{k,\beta}: \Omega^*(L)^{\otimes k} \lto \Omega^*(L) $$
 are maps of degree $2-k-\mu(\beta)$ and $1-k-\mu(\beta)$ for each $t\in [0,1]$. These have to satisfy
\begin{itemize}
\setlength{\parsep}{0pt}
\setlength{\itemsep}{0pt}
\item[{\bf(a)}] For $\xi_1,\ldots,\xi_k \in \Omega^*(L)$, $\m^t_{k,\beta}(\xi_1,\ldots,\xi_k), \mc^t_{k,\beta}(\xi_1,\ldots,\xi_k) \in \Omega^*([0,1]\times L)$;
 \item[{\bf(b)}] For each fixed $t$, $(\Omega^*(L), \m^t_{k,\beta})$ is an unital filtered \Ai-algebra with unit $1$;
 \item[{\bf(c)}] $\m^t_{k,0}$ is independent of $t$, $\mc^t_{k,0}=0$ and $\mc^t_{k,\beta}(\ldots,1,\ldots)=0$;
 \item[{\bf(d)}] \begin{align}\label{pisotopy}
0=(-1&)^{n+1} \frac{d}{dt}\m^t_{k,\beta}(\xi_1,\ldots,\xi_k)\nonumber\\
&-\sum_{\substack{\beta_1+\beta_2=\beta \\0\leq j\leq k\\1\leq i\leq k-j+1}}\m^t_{k-j+1,\beta_1}(\xi_1,\ldots,\mc^t_{j,\beta_2}(\xi_{i},\ldots,\xi_{i+j-1}),\ldots,\xi_k)\\
&+\sum_{\substack{\beta_1+\beta_2=\beta \\0\leq j\leq k\\1\leq i\leq k-j+1}}(-1)^{\sum_{l=1}^{i-1} ||\xi_l||}\mc^t_{k-j+1,\beta_1}(\xi_1,\ldots,\m^t_{j,\beta_2}(\xi_{i},\ldots,\xi_{i+j-1}),\ldots,\xi_k).\nonumber
\end{align}
 \item[{\bf(e)}] For all $k$ and $\beta$, $\m^{(0)}_{k,\beta}=\m^0_{k,\beta}$ and $\m^{(1)}_{k,\beta}=\m^1_{k,\beta}$.
\end{itemize}
With the obvious modifications we can also define pseudoisotopy modulo $T^E$.
\end{defn}

Fukaya shows that pseudoisotopy and pseudoisotopy modulo $T^E$ are equivalence relations. For our purposes the most important property is the following

\begin{thm}[{\cite[Theorem 8.1]{Fuk}}]\label{extendingmk}
 Let $E_0 < E_1$ and suppose $\m^{0}$ and $\m^{1}$ are \Ai-algebras modulo $T^{E_0}$ and modulo $T^{E_1}$ respectively, on $\Omega^*(L)$. If $(\m^t_{k,\beta}, \mc^t_{k,\beta})$ is a pseudoisotopy modulo $T^{E_0}$ between them, then we can extend $\m^{0}$ to an \Ai-algebra modulo $T^{E_1}$  $\m^{(1)}$. Moreover we can extend $(\m^t_{k,\beta}, \mc^t_{k,\beta})$ to a pseudoisotopy modulo $T^{E_1}$ between $\m^{(1)}$ and $\m^{1}$. 
\end{thm}
\begin{proof}
Since pseudoisotopy is a transitive relation it is enough to consider the case when $E(G)\cap [E_0,E_1] =\{E_0,E_1\}$.

If $E(\beta)\leq E_0$, $\m^t_{k,\beta}$ and $\mc^t_{k,\beta}$ are already defined. Otherwise we define $\mc^t_{k,\beta}=0$ and
\begin{align}\label{pisotopyformula}
\m^\tau_{k,\beta}&(\xi_1,\ldots,\xi_k)=\m^1_{k,\beta}(\xi_1,\ldots,\xi_k)\nonumber\\
&+(-1)^n\sum_{\substack{\beta_1+\beta_2=\beta}}\int_\tau^1\m^t_{k-j+1,\beta_1}(\xi_1,\ldots,\mc^t_{j,\beta_2}(\xi_{i},\ldots,\xi_{i+j-1}),\ldots,\xi_k)\\
&+(-1)^{n+1}\sum_{\substack{\beta_1+\beta_2=\beta}}(-1)^{\sum_{l=1}^{i-1} ||\xi_l||}\int_\tau^1\mc^t_{k-j+1,\beta_1}(\xi_1,\ldots,\m^t_{j,\beta_2}(\xi_{i},\ldots,\xi_{i+j-1}),\ldots,\xi_k)\nonumber.
\end{align}
Recall that $\mc_{k,0}=0$, therefore the right hand side of the above equation is well defined. We can then check that this defines a pseudoisotopy, we refer the reader to \cite{Fuk} for this.
\end{proof}

Now we go back the case of the Fukaya algebra. First consider $\mM_{k+1}^I(\beta)=[0,1]\times \mM_{k+1}(\beta)$ and denote by
\begin{align}
\tilde{ev}_i:&\mM_{k+1}^I(\beta)\lto\mM_{k+1}(\beta)\stackrel{ev_i}{\lto} L\nonumber\\
ev_t:&\mM_{k+1}^I(\beta)\lto[0,1]\nonumber
\end{align}
the natural projections and by $ev^I_i$ the product $ev_t\times \tilde{ev}_i$. Products of spaces with Kuranishi structures have natural Kuranishi structures (see \cite[Appendix A1.2]{FOOO}) therefore the spaces $\mM^I_{k+1}(\beta)$ have Kuranishi structures. Moreover there are forgetful maps $\forg_j : \mM^I_{k+1}(\beta) \lto \mM^I_{k}(\beta)$.
As in \cite{Joy}, we can use $ev^I_0$ to decompose the boundary:
\begin{align}\label{boundaryparameter}
 \partial \mM_{k+1}^I(\beta)&=\partial ^+ \mM_{k+1}^I(\beta)\cup\partial ^- \mM_{k+1}^I(\beta)\nonumber\\
\partial ^+\mM_{k+1}^I(\beta)&=\bigcup_{\substack{\beta'+\beta''=\beta\\0\leq j\leq k}}(-1)^{n+1+i(j+1)}\mM_{j+1}^I(\beta')\ _{ev^I_0}\times_{ev^I_i}\mM_{k-j+2}^I(\beta'')\\
\partial ^-\mM_{k+1}^I(\beta)&=(-1)\mM_{k+1}^0(\beta)\cup\mM_{k+1}^1(\beta).\nonumber
\end{align}

For our purposes, the main use of the spaces $\mM^I_{k+1}(\beta)$ is to interpolate between different choices of systems of multisections. The following proposition is proved by Fukaya in Section 11 of \cite{Fuk}.
\begin{prop}\label{oneparametersection} Fix $\epsilon, E > 0$ and let $S_0$ and $S_1$ be two systems of multisections on $\mM_{k+1}(\beta)$ for $\omega(\beta)\leq E$ satisfying the conditions of Proposition \ref{compmulti}. Then there exist continuous families of multisections on $\mM_{k+1}^I(\beta)$ such that
\begin{itemize}
\setlength{\parsep}{0pt}
\setlength{\itemsep}{0pt}
\item[{\bf(a)}] $S\vert_{t=0}=S_0$ and $S\vert_{t=1}=S_1$;
\item[{\bf(b)}] they are transversal and $\epsilon$-small;
\item[{\bf(c)}] the multisections are compatible with the forgetful maps $\forg_j$;
\item[{\bf(d)}] $ev^I_0\vert_{S^{-1}(0)}:S^{-1}(0)\lto [0,1]\times L$ is a submersion;
\item[{\bf(e)}] they are compatible with the boundary decomposition \eq{boundaryparameter}.
\end{itemize}
\end{prop}

Using this system of multisections, for $(k,\beta)\neq (1,0)$ we define
$$\m^I_{k,\beta}(\xi_1,\ldots,\xi_k)=(-1)^{\sum_{l=1}^k(k-l)\vert\xi_u\vert + \frac{(k-1)(k-2)}{2}}\Corr\left(ev^I_1\times\ldots\times ev^I_k,\mM_{k+1}^I(\beta)^S,ev^I_0\right)\big( \ \vec \xi \ \big)$$ 
for $(k,\beta)\neq (1,0)$ and $\m^I_{1,0}=(-1)^{n}d$. As before using the second part of Proposition \ref{kurstokes} and Proposition \ref{composition} we can show
\begin{prop}\label{ainfinityI} 
$(\Omega^*([0,1]\times L), \m^I_{k,\beta})$ is a filtered unital $A_\infty$-algebra modulo $T^E$.
\end{prop}

By definition of $\Omega^*([0,1]\times L)$, there are forms $\rho(t),\sigma(t)\in\Omega^*(L)$ such that 
$$\m^I_{k,\beta}(\xi_1,\ldots,\xi_k)=\rho(t) + dt\wedge\sigma(t).$$
We define
\begin{align}\label{mkt}
 \m^t_{k,\beta}(\xi_1,\ldots,\xi_k)=(-1)^k \rho(t), \ \ \mc^t_{k,\beta}(\xi_1,\ldots,\xi_k)=(-1)^k\sigma(t).
\end{align}
\begin{prop}\label{isotopy}
 Denote by $\m^{(0)}$ and $\m^{(1)}$ the \Ai-algebra structures on $\Omega^*(L)$ obtained from using the systems of multisections $S_0$ and $S_1$. Then $(\m^t_{k,\beta}, \mc^t_{k,\beta})$ as in (\ref{mkt}) define a pseudoisotopy between $\m^{(0)}$ and $\m^{(1)}$.
\end{prop}
\begin{proof}
First note that Proposition \ref{proppi*}(b) implies that 
$$\m^I_{k,\beta}(\xi_1,\ldots,\xi_{i-1},dt\wedge\xi(t),\xi_{i+1}\ldots,\xi_k)=(-1)^\epsilon dt\wedge\m^I_{k,\beta}(\xi_1,\ldots,\xi_{i-1},\xi(t),\xi_{i+1}\ldots,\xi_k),$$
for $\epsilon=k-i+|\xi_1|+\ldots +|\xi_{i-1}|+k$. Decomposing the \Ai-equations for $\m^I_{k,\beta}$ into sums with and without a $dt$ factor and using this identity we can easily prove that $(\m^t_{k,\beta}$ is an \Ai-algebra for each $t$ and that Definition \ref{defisotopy}(d) holds.

The other nontrivial properties we have to check are that $\m^{(0)}_{k,\beta}=\m^0_{k,\beta}$ and $\m^{(1)}_{k,\beta}=\m^1_{k,\beta}$. The proofs of both statements are the same so we sketch just the first one. Consider the following diagram
\begin{displaymath}
\xymatrix{ \mM_{k+1}(\beta) \ar[r]^{\tilde i_0} \ar[d]_{ev_i}& \mM_{k+1}^I(\beta)\ar[d]^{ev^I_i} \\
L\ar[r]^{i_0} & [0,1]\times L,}
\end{displaymath}
where $i_0$ and $\tilde i_0$ are the inclusions at $t=0$. Observe that $$\m^0_{k,\beta}(\xi_1,\ldots,\xi_k)=(-1)^k i_0^*\m^I_{k,\beta}(\xi_1,\ldots,\xi_k).$$ After adding the contributions from each Kuranishi neighborhood and restricting to the zero set of the multisection, the claim is equivalent to
$$i_0^* \circ (ev^I_0)_*= (-1)^k (ev_0)_*\circ \tilde i_0^*.$$
This follows from a small generalization of Proposition \ref{proppi*}(c), that is proved in \cite[Section 9.1]{GuiSte}. The sign is a consequence of the fact that induced orientations on the fibers of $ev_0$ and $ev^I_0$ differs by $(-1)^k$.
This shows $\m^{(0)}_{k,\beta}=\m^0_{k,\beta}$.
\end{proof}

Consider an increasing sequence $\{E_i\}_{i=1}^{\infty}$ such that $\lim_{i\to\infty}E_i= +\infty$. By Proposition \ref{aalgebra}, for each $i$ we define $\m^{i}$, an $A_\infty$-algebra modulo $T^{E_i}$ on $\Omega^*(L)$. By Proposition \ref{isotopy} there is $\mc^i$ a pseudoisotopy modulo $T^{E_i}$ between $\m^i$ and $\m^{i+1}$. Now we construct an \Ai-algebra by inductively extending $\m^{1}$ to an \Ai-algebra modulo $T^{E_n}$. Suppose that we have already constructed $\m^{(n)}$ an \Ai-algebra modulo $T^{E_n}$ and $\mc^{(n)}$ a pseudoisotopy modulo $T^{E_n}$ between $\m^{(n)}$ and $\m^{n+1}$. Then using Proposition \ref{extendingmk}, we extend $\m^{(n)}$ to an \Ai-algebra modulo $T^{E_{n+1}}$, which we denote by $\m^{(n+1)}$ and extend $\mc^{(n)}$ to a pseudoisotopy modulo $T^{E_{n+1}}$ between $\m^{(n+1)}$ and $\m^{n+1}$. We denote by $\mc^{(n+1)}$ the composition of this pseudoisotopy with $\mc^{n+1}$. Therefore $\mc^{(n+1)}$ is a pseudoisotopy modulo $T^{E_{n+1}}$ between  $\m^{(n+1)}$ and $\m^{n+2}$, 
which concludes the 
induction step. Finally we obtain an \Ai-algebra
$$(\Omega^*(L), \m^{(\infty)}),$$
which we define to be $\mF(L)$, the Fukaya algebra of $L$. We have the following

\begin{thm}[{\cite[Theorem 14.2]{Fuk}}] 
The unital filtered $A_\infty$-algebra $\mF(L)$ constructed above is independent of the choices of systems of multisections and almost complex structure up to filtered $A_\infty$-quasi-isomorphism.
\end{thm}

\section{$\mathcal{F}(L_1 \times L_2)$ modulo $T^E$}\label{secfukprod}

In this section we prove the modulo $T^E$ version of Theorem \ref{intcomsub}. As we saw in the previous section, $\mF(L)$ is constructed as a limit of \Ai-algebras modulo $T^E$, thus we will show that $\mF(L_1)$ and $\mF(L_2)$ are commuting subalgebras modulo $T^E$ of $\mF(L_1 \times L_2)$. We will prove this assuming the existence of some particular Kuranishi structures and systems of multisections on the moduli spaces $\mM_{k+1}(\beta_1 \times \beta_2)$. The proof of the existence of such Kuranishi structures and multisections will be given in \S \ref{proofs}. 

For $i=1,2$ consider $2n_i$-dimensional symplectic manifolds $(M_i,\omega_i)$ and compact Lagrangians $L_i\subseteq M_i$ with (relative) spin structures $\sigma_i$. Then the manifold $M=M_1\times M_2 $ has a symplectic form $\omega=p_1^*\omega_1+p_2^*\omega_2$, where $p_1:M_1\times M_2\lto M_1$ and $p_2:M_1\times M_2\lto M_2$ are the obvious projections, also $L=L_1\times L_2\subseteq M_1\times M_2$ is a Lagrangian. Note that $L$ is naturally oriented and it is relatively spin since $$w_2(TL)=w_2(T L_1)+w_2(TL_2)=i^*(\sigma_1)+i^*(\sigma_2)=i^*(\sigma_1+\sigma_2).$$
Additionally, $\sigma_1$ and $\sigma_2$ determine a choice $\sigma$ of relative spin structure on $L$.

We fix  almost complex structures $J_1$ and $J_2$ compatible with $\omega_1$ and $\omega_2$ respectively. Then $J=J_1\times J_2$ is an almost complex structure on $M$ compatible with $\omega$. Let $\Sigma$ be a bordered Riemann surface and consider a map $u:(\Sigma,\partial \Sigma)\lto (M_1\times M_2, L_1\times L_2)$. Denote by $u_1=p_1\circ u:(\Sigma,\partial \Sigma)\lto (M_1, L_1)$ and $u_2=p_2\circ u:(\Sigma,\partial \Sigma)\lto (M_2, L_2)$ the projections of $u$. We can easily check that
\begin{align}\label{partialbarJu}
\overline\partial_J u=\overline\partial_{J_1} u_1\times\overline\partial_{J_2} u_2.
\end{align}
Now, $\pi_2(M_1\times M_2, L_1\times L_2)=\pi_2(M_1, L_1)\oplus\pi_2(M_2,L_2)$ therefore we can rewrite $\beta\in\pi_2(M_1\times M_2, L_1\times L_2)$ as $\beta=\beta_1\times\beta_2$. From (\ref{partialbarJu}) we conclude that $\mM_1(\beta)\neq\emptyset$ if and only if $\mM_1(\beta_1)$ and $\mM_1(\beta_2)$ are non-empty. 

An easy computation shows that 
$$E(\beta_1\times\beta_2)=E(\beta_1)+E(\beta_2) \ \text{and} \  \mu(\beta_1\times\beta_2)=\mu(\beta_1)+\mu(\beta_2).$$ 
So we conclude that $G(L_1 \times L_2)=G(L_1) + G(L_2)$. We will use the following notation for the \Ai-operations on $\mF(L_1\times L_2)$. For each $\beta \in G(L_1 \times L_2)$ we take the decomposition
$$\m_{k,\beta}=\sum_{\substack{\beta_1\times\beta_2\in G(L_1)\times G(L_2)\\\beta_1+\beta_2=\beta}}\m_{k,\beta_1\times\beta_2},$$
where $\m_{k,\beta_1\times\beta_2}$ is defined using the moduli spaces $\mathcal{M}_{k+1}(\beta_1\times\beta_2)$.

Now we will define the maps between various moduli spaces, that we will use to prove all the vanishing conditions in the definitions of commuting subalgebras.
From the above discussion there are continuous maps 
$$\Pi_1:\mM_{k+1}(\beta_1\times\beta_2)\lto \mM_{k+1}(\beta_1) \ \ \text{and} \ \ \Pi_2:\mM_{k+1}(\beta_1\times\beta_2)\lto \mM_{k+1}(\beta_2)$$ 
given by $\Pi((\Sigma,\vec z),u)=((\Sigma_i,\vec z), u_i)$ where $\Sigma_i$ is obtained from $\Sigma$ by collapsing irreducible components that become unstable after forgetting the other component of $u$.
\begin{defn}\label{projections}
 Consider a decomposition of $\{ 1,\ldots, k\}$ into two disjoint sets $J$ and $L$. Let $\forg_J :\mM_{k+1}(\beta) \lto \mM_{k-|J|+1}(\beta)$ be the map that forgets the marked points $z_j$ for $j \in J$. We define
$$P_{J,L}: \mM_{k+1}(\beta_1\times\beta_2) \lto \mM_{k-|L|+1}(\beta_1) \times \mM_{k-|J|+1}(\beta_2),$$
as the composition $(\forg_J \times \forg_L)\circ(\Pi_1 \times \Pi_2)$.

Similarly, given a decomposition $\{ 1,\ldots, k+1\} \setminus \{i+1\}= J \sqcup L$ we define
$$Q_{i+1,J,L}: \mM_{k+2}(\beta_1\times\beta_2) \lto \mM_{k-|L|+2}(\beta_1) \times \mM_{k-|J|+2}(\beta_2),$$
as the composition $(\forg_J \times \forg_L)\circ(\Pi_1 \times \Pi_2)$.
\end{defn}

In order to distinguish the different evaluations maps we introduce the following notation. We denote by $Ev_i:\mM_{k+1}(\beta_1\times\beta_2)\lto L_1\times L_2$, $ev_i^1:\mM_{k+1}(\beta_1)\lto L_1$ and $ev_i^2:\mM_{k+1}(\beta_2)\lto L_2$ the usual evaluation maps. With this notation we have
$Ev_0= (ev_0^1 \times ev_0^2)\circ P_{J,L}$ and $Ev_l= (ev_l^1 \times ev_l^2)\circ Q_{i+1,J,L}$, for $l=0,i+1$.

\begin{prop}\label{compkuranishi}
 There are Kuranishi structures on $\mM_{k}(\beta_1 \times \beta_2)$, $\mM_{k}(\beta_l)$, $(l=1,2)$ and $\mM_{p}(\beta_1)\times\mM_{q}(\beta_2)$ for $p,q,k \geq 1$ satisfying
\begin{itemize}
\setlength{\parsep}{0pt}
\setlength{\itemsep}{0pt}
\item[{\bf(a)}] they are compatible, in the sense of Definition \ref{kurcompatible}, with $P_{J,L}$, $Q_{i+1,J,L}$ and $\forg_j$;
 \item[{\bf(b)}] the maps $Ev_0$, $ev_0^1$ and $ev_0^2$ are weak submersions;
 \item[{\bf(c)}] the decomposition of the boundary \eq{boundary} holds as spaces with oriented Kuranishi structures.
\end{itemize}
\end{prop}

\begin{prop}\label{b2=0} For the Kuranishi structures constructed in Proposition \ref{compkuranishi}, we have:
 $$\mM_{k+1}(\beta_1\times 0)=(-1)^{n_2k}\mM_{k+1}(\beta_1)\times L_2,$$
$$\mM_{k+1}(0\times\beta_2)=L_1\times\mM_{k+2}(\beta_2),$$ 
as spaces with oriented Kuranishi structures.
\end{prop}

\begin{prop}\label{compmultisection}
 Fix $\epsilon, E>0$, then there exist good coordinate systems and transversal, $\epsilon$-small continuous families of multisections $(W_\alpha, S_\alpha)$ on $\mM_{k}(\beta_1 \times \beta_2)$, $(W_{\alpha_l}, S_{\alpha_l})$ on $\mM_{k}(\beta_1)$, ($l=1,2$) and $(W_{\gamma}, S_{\gamma})$ on $\mM_{p}(\beta_1)\times\mM_{q}(\beta_2)$ for $p,q,k \geq 1$ and $E(\beta_1)+E(\beta_2) \leq E$ satisfying
\begin{itemize}
\setlength{\parsep}{0pt}
\setlength{\itemsep}{0pt}
\item[{\bf(a)}] they are compatible, in the sense of Definition \ref{corrcompatible} with $P_{J,L}$, $Q_{i+1,J,L}$ and $\forg_j$;
 \item[{\bf(b)}] the maps $(Ev_0)_\alpha\vert_{S^{-1}_\alpha(0)}$, $(ev_0^l)_{\alpha_l}\vert_{S^{-1}_{\alpha_l}(0)}$, $(l=1,2)$ and $(ev_0^1\times ev_0^2)_\gamma\vert_{S^{-1}_{\gamma}(0)}$ are submersions;
 \item[{\bf(c)}] the multisections are compatible with the boundary decomposition \eq{boundary}.
\end{itemize}
\end{prop}

We will postpone the proofs of these three propositions to the next section. Using the systems of multisections provided by these propositions we will show that $\mF(L_1)$ and $\mF(L_2)$ are commuting subalgebras modulo $T^E$ of $\mF(L_1 \times L_2)$. We start with the following 

\begin{prop}\label{b1b2nonzero}
 Consider $\beta_1 \in G(L_1)$ and $\beta_2 \in G(L_2)$ both non-zero satisfying $E(\beta_1)+E(\beta_2) \leq E$. Let $\xi$ and $\xi_1,\ldots, \xi_k \in \Omega^*(L_1\times L_2)$ be differential forms such that for each $i$, $\xi_i=p_1^*(a_i)$ or $\xi_i=p_2^*(b_i)$ for some forms $a_i\in \Omega^*(L_1)$ or $b_i\in \Omega^*(L_2)$. Then we have 
 $$\m_{k,\beta_1\times \beta_2}(\xi_1,\ldots,\xi_k)=0;$$
 $$\m_{k+1,\beta_1\times \beta_2}(\xi_1,\ldots,\xi_{i},\xi,\xi_{i+1},\ldots,\xi_k)=0.$$
\end{prop}
\begin{proof}
We define $J$ to be the set of $j$ such that $\xi_j=p_1^*(a_j)$ and $L=\{1,\ldots, k\} \setminus I$. The first statement follows from applying Proposition \ref{kurvanishingpi*} to the map $P_{J,L}$. To simplify the notation we will assume that $J=\{1,\ldots,l\}$, for some $l$. First, the dimension formula for the moduli spaces gives
\begin{align}
\vdim \mM_{k+1}(\beta_1 \times \beta_2)&=n_1 + n_2 + \mu(\beta_1\times\beta_2) +k-2\nonumber\\
&= n_1 + n_2 + \mu(\beta_1) + \mu(\beta_2) +k-2,\nonumber\\
\vdim \mM_{k-|L|+1}(\beta_1) \times \mM_{k-|J|+1}(\beta_2)&=n_1 + \mu(\beta_1) + k- |L| -2 + n_2 \nonumber\\
& \ \ \ +\mu(\beta_2) +k-|J|-2\\
&= n_1 +n_2 + \mu(\beta_1) + \mu(\beta_2) +k-4.\nonumber
\end{align}
Ignoring the signs introduced in the definition of the maps $\m_{k,\beta}$, we have 
\begin{align}
\m_{k,\beta_1 \times \beta_2}(\xi_1,&\ldots,\xi_k)=\Corr\big(Ev_1\times\ldots\times Ev_k,\mM_{k+1}(\beta_1\times\beta_2)^S,Ev_0\big)(\xi_1\times\ldots\times\xi_k)\nonumber\\
&=\Corr\big(ev,\mM_{k+1}(\beta_1\times\beta_2)^S,Ev_0\big(a_1 \times\ldots a_l \times b_{l+1} \times\ldots\times b_k).\label{mkb1b2}
\end{align}
where $ev:\mM_{k+1}(\beta_1\times\beta_2) \lto L_1^{\times l}\times L_2^{\times (k-l)}$ is defined as 
$$ev=(p_1 \circ Ev_1)\times \ldots \times (p_1 \circ Ev_l) \times (p_2 \circ Ev_{l+1})\times \ldots \times (p_2 \circ Ev_k),$$
which satisfies $ev=ev_1^1\times\ldots\times ev_l^1 \times ev_{l+1}^2\times\ldots\times ev_{k}^2 \circ P_{J,L}$. Since we have chosen a compatible system of multisections, we can apply Proposition \ref{kurvanishingpi*} to conclude that (\ref{mkb1b2}) vanishes.

The second statement is analogous. We define $J$ and $L$ in the same way, but this time use the map $Q_{i+1,J,L}$. Counting dimensions as before, we have 
$$\vdim \mM_{k+2}(\beta_1\times\beta_2)=\vdim \mM_{k-|L|+2}(\beta_1) \times \mM_{k-|J|+2}(\beta_2) +1.$$
Again, to simplify the notation, we assume that $J=\{1,\ldots,i\}$. Then
\begin{align}
\m_{k+1,\beta_1 \times \beta_2}(&\xi_1,\ldots,\xi,\ldots,\xi_k)=\Corr\big(Ev_1\times\ldots\times Ev_{k+1},\mM_{k+2}(\beta_1\times\beta_2)^S,Ev_0\big)\big( \ \vec\xi \ \big)\nonumber\\
&=\Corr\big(ev,\mM_{k+2}(\beta_1\times\beta_2)^S,Ev_0)\big(a_1 \times\ldots a_i \times \xi \times b_{i+1} \times\ldots\times b_k)\label{mk+1b1b2}
\end{align}
where now, $ev:\mM_{k+1}(\beta_1\times\beta_2) \lto L_1^{\times l}\times (L_1 \times L_2) \times L_2^{\times (k-l)}$ is defined as 
$$ev=(p_1 \circ Ev_1)\times \ldots \times (p_1 \circ Ev_i) \times Ev_{i+1} \times (p_2 \circ Ev_{i+2})\times \ldots \times (p_2 \circ Ev_{k+1}).$$
As above, Proposition \ref{kurvanishingpi*} implies that (\ref{mk+1b1b2}) vanishes.
\end{proof}

\begin{prop}\label{mkb10=0}
 Consider $\beta_1 \in G(L_1)$ and $\beta_2 \in G(L_2)$ both non-zero satisfying $E(\beta_1)$, $E(\beta_2) \leq E$ and let $\xi_1,\ldots, \xi_k \in \Omega^*(L_1\times L_2)$ be differential forms. If $\xi_i=p_2^*(b_i)$, then
 $$\m_{k,\beta_1\times 0}(\xi_1,\ldots, \xi_k)=0.$$
 If $\xi_i=p_1^*(a_i)$, then
 $$\m_{k,0\times \beta_2}(\xi_1,\ldots, \xi_k)=0.$$
\end{prop}
\begin{proof}
The proof of both statements is essentially the same, so we will carry out only the first one. Ignoring the signs, we have
\begin{align}
\m_{k,\beta_1 \times 0}(\xi_1,&\ldots,\xi_k)=\Corr\big(Ev_1\times\ldots\times Ev_{k},\mM_{k+1}(\beta_1\times 0)^S,Ev_0\big)(\xi_1\times\ldots\times\xi_k)\nonumber\\
&=\Corr\big(ev,\mM_{k}(\beta_1\times 0)^S,Ev_0\big)\big(\xi_1 \times\ldots \times b_i \times \xi_{i+1} \times\ldots\times \xi_k).\label{mkb10}
\end{align}
This time $ev=Ev_1 \times\ldots\times q \times\ldots Ev_k$, where $q$ is defined as the projection $q: \mM_{k+1}(\beta_1\times 0) \simeq \mM_{k+1}(\beta_1) \times L_2 \lto L_2$.
The map $ev$ commutes with the forgetful map $\forg_i: \mM_{k+1}(\beta_1 \times 0) \lto \mM_{k}(\beta_1 \times 0)$. Since $\vdim \mM_{k+1}(\beta_1 \times 0)>\vdim \mM_{k}(\beta_1 \times 0)$, we can apply Proposition \ref{kurvanishingpi*} to conclude that (\ref{mkb10}) vanishes.
\end{proof}

We are now ready to prove the modulo $T^E$ version of Theorem \ref{intcomsub}.

\begin{prop}\label{comsubalgmodte}
 Fix $E>0$, then $\mF(L_1)$ and $\mF(L_2)$ are commuting subalgebras $mod \ T^E$ of $\mF(L_1\times L_2)$ via the inclusions
  \begin{align}
  \begin{array}{lll}
\iota :\mF(L_1) \lto \mF(L_1\times L_2), & & \iota : \mF(L_2) \lto \mF(L_1\times L_2)\\
\xi_1 \lto (-1)^{|\xi_1|n_2} p_1^*(\xi_1)                            & &  \xi_2 \lto (-1)^{|\xi_2|n_1} p_2^*(\xi_2),\\
\end{array}\nonumber
 \end{align}
\end{prop}
\begin{proof}
The previous two propositions prove all the vanishing conditions on Definitions \ref{subalgebra} and \ref{comsubalg}. All that is left to show that $\mF(L_1)$ and $\mF(L_2)$ are subalgebras modulo $T^E$ of $\mF(L_1\times L_2)$, are the following two equalities,
\begin{align}
\m_{k,\beta_1 \times 0}(\iota(a_1),\ldots,\iota(a_k))&=\iota(\m_{k,\beta_1}(a_1,\ldots,a_k))\nonumber\\
\m_{k,  0\times\beta_2}(\iota(b_1),\ldots,\iota(b_k))&=\iota(\m_{k,\beta_2}(b_1,\ldots,b_k)).\nonumber
\end{align}
For $(k,\beta)=(1,0)$, this simply follows from the definitions, in the other cases we argue as follows. The \Ai-operations are defined as a smooth correspondence on a Kuranishi space with a continuous family of multisections. These are defined by summing the contributions of each Kuranishi neighborhood. Thus it is enough to check that the above equalities hold on each Kuranishi neighborhood in $\mM_{k+1}(\beta_1\times 0)$, which are in correspondence with Kuranishi neighborhoods in $\mM_{k+1}(\beta_1)\times L_2$.  Propositions \ref{b2=0} and \ref{compmultisection} guarantee that for each Kuranishi neighborhood, $S_{\a,i,j}^{-1}(0)\simeq (-1)^{k n_2} S_{\a,i,j}^{-1}(0) \times L_2$, as oriented smooth manifolds and moreover, $Ev_i=ev_i^1 \times id$. To simplify the notation, we will assume there is only one neighborhood and compute
\begin{align}
 \m_{k,\beta_1 \times 0}(\iota(a_1),\ldots,\iota(a_k))&=(-1)^{\gamma+ n_2 k} (ev_0 \times id)_*(ev_1^*(a_1)\wedge \ldots \wedge ev_k^*(a_k)\wedge \theta)\nonumber\\
 &=(-1)^{\gamma+ n_2 k} p_1^*(ev_0)_*(ev_1^*(a_1)\wedge \ldots \wedge ev_k^*(a_k)\wedge \theta),
\end{align}
by Proposition \ref{proppi*}(d). Above by definition, we have $\gamma=\sum_i (k-i)|a_i| +\frac{(k-1)(k-2)}{2} + n_2 \sum_i |a_i|$, hence
\begin{align}
 \m_{k,\beta_1 \times 0}(\iota(a_1),\ldots,\iota(a_k))&=(-1)^{n_2(k + \sum_i |a_i|)} p_1^*(\m_{k,\beta_1}(a_1,\ldots,a_k))\nonumber\\
 &=\iota(\m_{k,\beta_1}(a_1,\ldots,a_k)).\nonumber
\end{align}
This finishes the proof that $\mF(L_1)$ is subalgebra modulo $T^E$ of $\mF(L_1\times L_2)$. The proof for $\mF(L_2)$ is completely analogous.

Next we need to check the commuting relations in Definition \ref{comsubalg}. The first condition follows from commutativity of the wedge product of forms once we note that $\m_{2,0}(\xi_1,\xi_2)=(-1)^{|\xi_1|}\xi_1 \wedge \xi_2$. The second condition follows from combining Proposition \ref{b1b2nonzero} and the above computation of $\m_{k,\beta_1\times 0}$ and $\m_{k,0\times\beta_2}$ in the case $k=0$. Definition \ref{comsubalg}(c) follows from Propositions \ref{b1b2nonzero} and \ref{mkb10=0}, in the cases when $\mu_{k+1,\beta}$ should vanish. We need to check the exceptions in Definition \ref{comsubalg}(c)(i) and (ii), these are equivalent to 
\begin{align}
\m_{k+1,\beta_1\times 0}(\iota(a_1),\ldots,\iota(a_{i}),&K(a\otimes b),\iota(a_{i+1}),\ldots\iota(a_k))\nonumber\\
&=(-1)^{|b|\sum_{j> i}\|a_j\|}K(\m_{k+1,\beta_1}(a_1,\ldots,a,\ldots,a_k)\otimes b)\nonumber\\
\m_{k+1,0 \times \beta_2}(\iota(b_1),\ldots,\iota(b_{i}),&K(a\otimes b),\iota(b_{i+1}),\ldots\iota(b_k))\nonumber\\
&=(-1)^{|a|\left(1+\sum_{j\leq i}\|b_j\|\right)}K(a\otimes \m_{k+1,\beta_2}(b_1\ldots,b,\ldots,b_k)).\nonumber
\end{align}
When $k=0$ and $\beta_1=\beta_2=0$, these are consequence of the Leibniz rule for the de Rham differential. In the remaining cases we follow the same strategy as above. We spell out the proof of the second equality. Assuming the same simplifications and using Proposition \ref{b2=0} we compute
\begin{align}
 &\m_{k+1, 0 \times \beta_2}(\iota(b_1),\ldots,\iota(b_{i}),K(a\otimes b),\iota(b_{i+1}),\ldots\iota(b_k))\nonumber\\
 &=(-1)^{\gamma_1} (id \times ev_0)_*(ev_1^*(b_1)\wedge \ldots \wedge a\wedge ev_{i+1}^*(b)\wedge \ldots \wedge ev_{k+1}^*(b_k)\wedge \theta)\nonumber\\
 &=(-1)^{\gamma_2} (id \times ev_0)_*(a\wedge ev_1^*(b_1)\wedge \ldots \wedge ev_{i+1}^*(b)\wedge \ldots \wedge ev_{k+1}^*(b_k)\wedge \theta)\nonumber\\
 &=(-1)^{\gamma_3} p_1^*(a)\wedge p_2^*((ev_0)_*(ev_1^*(b_1)\wedge \ldots \wedge ev_{i+1}^*(b)\wedge \ldots \wedge ev_{k+1}^*(b_k)\wedge \theta)\nonumber\\
 &=(-1)^{\gamma_4}p_1^*(a)\wedge p_2^*(\m_{k+1,\beta_2}(b_1\ldots,b,\ldots,b_k))\nonumber\\
 &=(-1)^{\gamma_5}K(a\otimes \m_{k+1,\beta_2}(b_1\ldots,b,\ldots,b_k))\nonumber
 \end{align}
where the signs are as follows
\begin{align}
\gamma_1&= \sum_{l\leq i} (k+1-l)|b_l| + (k-i)(|a|+|b|)+\sum_{l> i} (k-l)|b_l|+\frac{k(k-1)}{2}\nonumber\\
& \ \ \ +n_2|a|+n_1(|b| + \sum_l |b_l|),\nonumber\\
\gamma_2&= \gamma_1 + |a|\sum_{l\leq i}|b_l|,\nonumber\\
\gamma_3&= \gamma_2 +(k+1)(|a|+n_1) ,\nonumber\\
\gamma_4&= \gamma_3 + \sum_{l\leq i} (k+1-l)|b_l| + (k-i)|b|+\sum_{l> i} (k-l)|b_l|+\frac{k(k-1)}{2},\nonumber\\
\gamma_5&= \gamma_4 + n_2|a|+n_1(|b| + \sum_l |b_l|+k+1).\nonumber
\end{align}
Here the signs $\gamma_1$ and $\gamma_5$ are given by the definitions of the $\m_k$, $\iota$ and $K$, plus the fact that $\m_k$ has degree $k \pmod 2$. The sign $\gamma_3$ is given by Proposition \ref{proppi*}(d). An elementary computation shows
$$\gamma_5= |a|(1+ \sum_{l\leq i}||b_l||),$$
which implies the result.
\end{proof}

\section{The moduli spaces $\mM_{k+1}(\beta_1\times\beta_2)$}\label{proofs}
In this section, we prove Propositions \ref{compkuranishi}, \ref{b2=0} and \ref{compmultisection}. The proofs will follow the strategy in \cite{Fuk}. We construct Kuranishi neighborhoods and multisections by induction on the area of the disks. This guarantees that the boundary decompositions (\ref{boundary}) of the various moduli spaces are respected. Moreover to ensure compatibility with the forgetful maps $\forg_i$ we first construct Kuranishi structures (and systems of multisections) on the moduli spaces $\mM_1(\beta)$ and then pull them back to $\mM_{k+1}(\beta)$. In fact, Fukaya constructs Kuranishi structures in $\mM_0(\beta)$ first. This is needed only if one wants to make $\mF(L)$ a cyclic \Ai-algebra. As we don't make any use of the cyclic structure we will ignore this point. Finally, to guarantee compatibility with the maps $P_{J,L}$ and $Q_{i+1,J,L}$ we first construct Kuranishi neighborhoods (and systems of multisections) on the moduli spaces $\mM_1(\beta_1)\times \mM_1(\beta_2)$ and pull these back to $\mM_1(\beta_1\times \beta_2)$ via the map $\Pi_1 \times \Pi_2$. Then 
we use the forgetful maps to pull back these Kuranishi structures to $\mM_{k+1}(\beta_1\times\beta_2)$ and  $\mM_p(\beta_1)\times \mM_q(\beta_2)$, for $k=p+q$ or $k=p+q-1$.

\subsection{Proof of Proposition \ref{compkuranishi}}\label{proof1} We first review Fukaya's construction of Kuranishi structures on $\mM_{k+1}(\beta)$ for $k\geq0$ and $\beta=\beta_1$ or $\beta_2$ compatible with $\forg_i$, following \cite{Fuk}.

Let $\forg: \mM_{k+1}(\beta) \lto \mM_1(\beta)$ be the maps that forgets all the marked points except the first one. Consider $(\Sigma, u)\in\mM_{k+1}(\beta)$ and $\forg(\Sigma, u)=(\Sigma^0, u^0)\in\mM_{1}(\beta)$. Denote by $\Gamma$ and $\Gamma^0$ be the (finite) groups of automorphisms of $(\Sigma, u)$ and $(\Sigma^0, u^0)$ respectively. The map $\forg$ induces a map $h:\Gamma \lto \Gamma^0$. In this case, where there is at least one boundary marked in both curves, $h$ is an isomorphism, in the general case, where we forget all the marked points, it is injective. 

Consider the decomposition $\Sigma^0=\cup_{a\in A} \Sigma_a^0$ of $\Sigma^0$ into irreducible components (spheres or disks). Given an irreducible component $\Sigma_a^0$, denote by $k_a$ the number of boundary special  (either marked or singular) points and by $l_a$ the number of interior special points. We say $\Sigma_a^0$ is \emph{stable} if $k_a+2l_a\geq 3$ when $\Sigma_a^0$ is a disk or $l_a\geq 3$ when $\Sigma_a^0$ is a sphere. When $\Sigma_a^0$ is unstable we add enough interior marked points to make it stable and denote the new curve by $\Sigma^{0,+}$. Without loss of generality we can assume that for each point $w$ we add, the map $u^0\vert_{\Sigma_a^0}$ is an immersion at $w$, since, by stability of $(\Sigma^0,u^0)$, $u^0\vert_{\Sigma_a^0}$ is not constant when $\Sigma_a^0$ is unstable. Furthermore, we can pick the set of added marked points $\nu(\Sigma^{0,+})$ so that $\Gamma^0$ acts freely on $\nu(\Sigma^{0,+})$. We make one additional choice, for each $w\in\nu(\Sigma^{0,+})$ we pick a $(2n-2)$-
dimensional submanifold of $M$, $N_w^{2n-2}\subset M$, such that $N_w$ and $u^0\vert_\Sigma$ intersect transversely at $u(w)$. Additionally we require that $N_w=N_{\gamma(w)}$ for any $\gamma\in\Gamma^0$. We stabilize $\Sigma$ by adding  the same marked points. 

Next we have to construct the obstruction bundles. 
\begin{lem}[{\cite[Lemma 3.1]{Fuk}}]\label{0bundle}
Let $\Sigma_a^0$ be one of the irreducible components of $\Sigma^0$. Assume $\Sigma_a^0=D^2$ is a disk and let $U\subseteq D^2$ be any open subset of the interior of $D^2$, that does not intersect small balls around the interior special points of $\Sigma^{0,+}$. Then there exists a finite dimensional subspace $E(u)\subseteq L^p(\Sigma,u^*TM\otimes\Lambda^{0,1})$ satisfying:
\begin{itemize}
\setlength{\parsep}{0pt}
\setlength{\itemsep}{0pt}
\item[{\bf(a)}] elements of $E(u)$ are smooth and supported in U;
\item[{\bf(b)}] $\textrm{Im } D_u\bar\partial+ E(u)=L^p(\Sigma,u^*TM\otimes\Lambda^{0,1})$;
\item[{\bf(c)}] if $K(u)=(D_u\bar\partial)^{-1}(E(u))$ then, for the first boundary marked point $z_0$, $Ev_{z_0}:K(u)\lto T_{u(z_0)}L$ defined by $Ev_{z_0}(\xi)=\xi(z_0)$ is surjective.
\end{itemize}
\end{lem}
For the case of spheres there is a similar statement. We refer the reader to \cite[Lemma 3.2]{Fuk} for the full details.

Going back to the previous situation, we fix spaces $E_a$ provided by lemma for each component of $\Sigma^{0,+}$. We use the same spaces for the corresponding components of $\Sigma^+$ and take $E_a=0$ for the components which are contracted by $\forg$. By construction $\Sigma^+\in\mM_{\vert\nu(\Sigma^+)\vert,k+1}$ and $\Sigma^{0,+}\in\mM_{\vert\nu(\Sigma^{0,+})\vert,1}$. These are orbifold with corners, so we can find neighborhood of $\Sigma^+$ and $\Sigma^{0,+}$
$$U(\Sigma^+)=\frac{V(\Sigma^+)}{Aut(\Sigma^+)},\ \ U(\Sigma^{0,+})=\frac{V(\Sigma^{0,+})}{Aut(\Sigma^{0,+})},$$ 
where  $V(\Sigma^+), V(\Sigma^{0,+})$ are manifolds with corners and $Aut$ are the (finite) group of automorphisms of the corresponding curve.  Consider $\eta\in V(\Sigma^+)$ and let $\Sigma^+(\eta)$ be the corresponding Riemann surface. For maps 
$$v:(\Sigma^+(\eta),\partial \Sigma^+(\eta))\lto (M,L)$$  
which are $\epsilon$-close to $u$ (see \cite{Fuk} for the definition), for sufficiently small $\epsilon$, we can regard $E_a$ as a subspace of $L^p(\Sigma^+(\eta),v^*TM\otimes\Lambda^{0,1})$. Namely we take parallel transport along the minimal geodesic from $u(x)$ to $v(x)$, for each $x\in U_a$, and denote the result by $E_a(\Sigma^+,v)$. We denote $E(\Sigma^+,v)=\oplus_a E_a(\Sigma^+,v)$  and define
$$V^+(\Sigma,u)=\big\{(\eta, v)\big\vert v:(\Sigma^+(\eta),\partial \Sigma^+(\eta))\rightarrow (M,L)\textrm{ is $\epsilon$-close to $u$ and }\bar\partial_J v\in E(\Sigma^+,v)\big\}.$$ 
Similarly we can define $V^+(\Sigma^0,u^0)$. We have the following:
\begin{prop}[{\cite[Section 7.1]{FOOO}}]\label{gluinganalysis} If $\epsilon$ is sufficiently small, $V^+(\Sigma,u)$ and $V^+(\Sigma^0,u^0)$ are manifolds with corners.
\end{prop}
From the choice of $\nu(\Sigma^+)$ the map
\begin{align}
ev:V^+(\Sigma,u)&\lto M^{\vert\nu(\Sigma^+)\vert}\nonumber\\
(\eta,v)&\longmapsto (v(w_1),\ldots,v(w_{\vert\nu(\Sigma^+)\vert})),\nonumber
\end{align} is transversal to $N_{w_1}\times\ldots\times N_{w_{\vert\nu(\Sigma^+)\vert}}$. Therefore $$V(\Sigma,u)=V^+(p)\times_{M^{\vert\nu(\Sigma^+)\vert}}\prod_{w\in\nu(\Sigma^+)}N_w$$ is a manifold with corners.
Then we define a Kuranishi neighborhood of $p=(\Sigma,u)$ as follows:
$$\mU_p=\Big(V_p=V(\Sigma,u),\ E_p=\bigoplus_{a} E_a(\Sigma^+,u),\ s_p(\eta,v)=\bar{\partial}_Jv,\ \Gamma,\ \psi_p\Big).$$
If $s_p(\eta,v)=0$, we define
$$\psi(\eta,v)=(\Sigma(\eta),v)\in \mM_{k+1}(\beta)$$ where $\Sigma(\eta)$ is the surface $\Sigma^+(\eta)$ with the points in $\nu(\Sigma^+)$ removed.
In a similar way we define the manifold with corners $V(\Sigma^0,u^0)$ and construct a Kuranishi neighborhood for $q=(\Sigma^0,u^0)$ as
$$\mU_q=\Big(V_q=V(\Sigma^0,u^0),\ E_q=\bigoplus_{a} E_a(\Sigma^{0,+},u),\ s_p(\eta,v)=\bar{\partial}_Jv,\ \Gamma^0,\ \psi_q\Big).$$
From the construction there is a map 
$$\varphi: V_p \lto V_q,$$
which is locally modeled on the forgetful map $\mM_{l,k+1}\lto\mM_{l,1}$. However, to construct coordinate transformations between different Kuranishi neighborhoods, we have to take a special smooth structure on these moduli spaces, see \cite[Appendix A.1.4]{FOOO} for details on this.  Due to this particular choice of coordinates the map $\varphi$ is continuous but not smooth. In fact, the manifold $V_p$ can be stratified according to the combinatorial type of the underlying curve and when we restrict to one of these strata, $\varphi$ is a smooth submersion. In particular $\varphi$ is a smooth submersion when restricted to the top dimensional strata (which is open and dense). Nevertheless, as explained in \cite[Appendix A.1.4]{FOOO}, $s_p=s_q \circ \varphi$ is still a smooth map.

To complete the proof we perform induction on $E(\beta)$ to construct Kuranishi neighborhoods on $\mM_1(\beta)$ and use the previous argument to obtain compatible Kuranishi neighborhoods on $\mM_{k+1}(\beta)$ for all $k\geq0$. We take $\beta_1$ with minimal area $\omega(\beta_1)$ and consider $p=(\Sigma,u)\in\mM_1(\beta_1)$. Following the above procedure we obtain Kuranishi structures on $\mM_{1}(\beta_1)$. Observe that, by Lemma \ref{0bundle}(c), $ev_0:\mM_1(\beta_1)\lto L$ is weakly submersive. Assume now, by induction, that we have constructed the required Kuranishi structures on $\mM_1(\beta')$ for all $\beta'$ such that $\omega(\beta')<\omega(\beta)$. We consider one possible boundary strata component of $\mM_1(\beta)$, all the other are similar:
$$\mM_2(\beta_1)\ _{ev_0}\times_{ev_0}\mM_1(\beta_2),$$
where $\beta_1+\beta_2=\beta$. By induction both factors already have Kuranishi structures, so we take the fiber product Kuranishi structure. The map
$$ev_0:\mM_2(\beta_1)\ _{ev_1}\times_{ev_0}\mM_1(\beta_2)\lto L$$ which is induced by $ev_0:\mM_2(\beta_1)\lto L$, is weakly submersive by \cite[Lemma 2.2]{Fuk}, since by induction $ev_0$ on both components is weakly submersive. Therefore, as in Proposition \ref{gluinganalysis} we can construct a  Kuranishi structure on a neighborhood of this component extending the structure on the boundary, so that $ev_0:\mM_1(\beta)\lto L$ is still submersive. Which completes the inductive step.

Next we construct the Kuranishi structures on the moduli spaces $\mM_1(\beta_1)\times \mM_1(\beta_2)$ and $\mM_1(\beta_1\times \beta_2)$ compatible with  the map $\Pi_1 \times \Pi_2$. The strategy is similar to the one above. Consider $(\Sigma, u) \in \mM_1(\beta_1\times \beta_2)$ and let $(\Sigma^1, u_1)=\Pi_1(\Sigma, u)$ and $(\Sigma^2, u_2)=\Pi_2(\Sigma, u)$. 

First to stabilize the domains, we proceed as follows. As before we pick interior marked points $\nu(\Sigma^{1,+})$ and submanifolds $N^1_{w^1}$ of $M_1$ so that each irreducible component of $\Sigma^{1,+}$ is stable and $N^1_{w^1}$ is transversal to $u_1$ at $w^1$, for each $w^1\in\nu(\Sigma^{1,+})$. We do the same for $(\Sigma^2,u_2)$. From the definition of $\Sigma^1$ and $\Sigma^2$ we can regard elements of $\nu(\Sigma^{1,+})$ and $\nu(\Sigma^{2,+})$ as points in $\Sigma$. Without loss of generality we can assume that $\nu(\Sigma^{1,+})\cap\nu(\Sigma^{2,+})=\emptyset$. Then we define $\Sigma^+$ to be the curved obtained from $\Sigma$ by adding the marked points $\nu(\Sigma^+)=\nu(\Sigma^{1,+})\cup \nu(\Sigma^{2,+})$. Note that $\Sigma^+$ is a stable curve and that, for points $w\in \nu(\Sigma^{1,+})$, $u$ is transversal to $N_w:=N_w^1\times M_2$ at $w$ and for points $w\in\nu(\Sigma^{2,+})$, $u$ is transversal to $N_w:=M_1\times N_w^2$.

We can easily see that there are maps $\Gamma=Aut(\Sigma,u)\lto\Gamma_i=Aut(\Sigma^i,u_i)$, for $i=1,2$, so that the product $h:\Gamma\lto\Gamma_1\times\Gamma_2$ is an injection. Thus, since we took $\nu(\Sigma^{1,+})$ and $\nu(\Sigma^{2,+})$ to be invariant under $\Gamma_1$ and $\Gamma_2$ (respectively), we have that $\nu(\Sigma^{+})=\nu(\Sigma^{1,+})\cup\nu(\Sigma^{2,+})$ is $\Gamma$-invariant.

Next we choose the obstruction bundles. For each irreducible component of $\Sigma^1$ and $\Sigma^2$, we choose spaces $E(u_1)$ and $E(u_2)$ satisfying the conclusions of Lemma \ref{0bundle} for the maps $(\Sigma_a,u_1\vert_{\Sigma_a})$ and $ (\Sigma_a,u_2\vert_{\Sigma_a})$ respectively. In the case $u_2\vert_{\Sigma_a}$ (respectively $u_1\vert_{\Sigma_a}$) is constant, we simply take $E(u_2\vert_{\Sigma_a})=0$ (respectively $E(u_1\vert_{\Sigma_a})=0$).
Then the space $E(u\vert_{\Sigma_a})=E(u_1\vert_{\Sigma_a})\oplus E(u_2\vert_{\Sigma_a})$ satisfies the conditions of the lemma for $(\Sigma_a,u\vert_{\Sigma_a})$, namely:
\begin{align}
\textrm{Im} D_u\bar\partial+E(u)&=(\textrm{Im} D_{u_1}\bar\partial\oplus \textrm{Im} D_{u_2}\bar\partial)+(E(u_1)\oplus E(u_2))\nonumber\\
&=(\textrm{Im} D_{u_1}\bar\partial+E(u_1))\oplus (\textrm{Im} D_{u_2}\bar\partial+E(u_2))\nonumber\\
&=L^p(\Sigma_a,u_1^*TM_1\otimes\Lambda^{0,1})\oplus L^p(\Sigma_a,u_2^*TM_2\otimes \Lambda^{0,1})\nonumber\\
&= L^p(\Sigma_a,u^*TM\otimes\Lambda^{0,1}),\nonumber
\end{align}
and $K(u)=K(u_1)\oplus K(u_2)$. Then for any $z_0\in\partial\Sigma_a$
$$Ev_{z_0}:K(u)\lto T_{u(z_0)}L=T_{u_1(z_0)}L_1\oplus T_{u_2(z_0)}L_2$$
splits as $Ev_{z_0}=Ev^1_{z_0}\oplus Ev^2_{z_0}$, where $Ev^i_{z_0}=(p_i)_*\circ Ev_{z_0}$ for $p_{i*}:T_*(L_1\times L_2)\lto TL_1\oplus TL_2$.
Since each $Ev^i_{z_0}$ is surjective, we conclude $Ev_{z_0}$ is surjective. The situation of where the component is contracted by $\Pi_1$ (or $\Pi_1$) can be handed similarly. 

However this choice of spaces $E_a$ can fail to satisfy the support condition in Lemma \ref{0bundle}~(a), that requires that elements of $E_a$ be supported away from (interior) special points. This can happen only in the case we now describe: suppose $\Sigma_a$ is a component of $\Sigma$ which is not contracted by $\Pi_1$ or $\Pi_2$ but there is a tree of sphere bubbles $\Theta$ attached to some interior point $y \in \Sigma_a$ which is constant in the second factor and therefore contracted by $\Pi_2$ (there is also the symmetric case, where the bubbles are constant in the first factor, but this can be handled in the same way). Then the corresponding point $y \in \Sigma_a^2$ is not a special point in $\Sigma_a^2$ since the sphere components were contracted. And so we have $E(u\vert_{\Sigma_a})=E(u_1\vert_{\Sigma^1_a})\oplus E(u_2\vert_{\Sigma^2_a})$, where elements of $E(u_1\vert_{\Sigma^1_a})$ are supported away from $y$, but elements in $E(u_2\vert_{\Sigma^2_a})$ might not be.

Nevertheless the rest of the construction can be carried out with this choice of $E(u)$. The reason is the following: this support condition, imposed in \cite{Fuk,FOOO}, is used to perform the gluing near the resolution of the node $y$ in the curve $\Sigma$. In our case the target $M$ is a product $M_1\times M_2$ and we are using a product almost complex structure (and a product metric) therefore the gluing is done factor-wise. Moreover, since the second factor of the map is constant in the sphere components, we are actually not performing any gluing (of maps) in the second factor, only on the first where we do have the usual support condition.

For purposes of clarity, we provide all the details to the above argument in the simpler case where $\Sigma$ has only two irreducible components $\Lambda$ and $\Theta$ (the sphere). Denote by $v=(v_1,v_2)$ the restriction of the map $u$ to $\Lambda$ and by $w=(w_1,w_2)$ the restriction to $\Theta$. By assumption $w_2$ is constant, therefore we can describe $V^+(\Sigma,u)$ as follows:
$$\displaystyle V^+(\Sigma,u)=\left\{(\eta, t) \Big\vert \begin{array}{ll}
t=(t_1,t_2):(\Sigma^+(\eta), \ \partial \Sigma^+(\eta))\rightarrow (M,L) \textrm{ is $\epsilon$-close to $u$}, \\
 \bar{\partial}_{J_1} t_1\in E(\Sigma_1^+, u_1), \  t_2= \tilde{t}_2\circ p,  \ \bar{\partial}_{J_2} \tilde{t}_2\in E(\Lambda^+, v_2),
\end{array}
\right\},$$
where $p:\Sigma^+(\eta) \to \Sigma_2^+(\eta)  $ is the collapse map. We can prove the above is a smooth manifold with corners, in the same way as Proposition \ref{gluinganalysis}. This is because $\Sigma_2^+(\eta)$ is a smooth disk, hence we
only need to consider the gluing problem in the first factor, where we use $E(\Sigma_1^+, u_1)$ which has the correct support condition so the results of \cite{FOOO} do apply.
Therefore the fact that elements in $E(u_2\vert_{\Lambda})$ may be supported near $y$ causes no problem.

Now we can follow the same procedure as before to construct Kuranishi structures on $\mM_1(\beta_1)\times \mM_1(\beta_2)$ and $\mM_1(\beta_1\times \beta_2)$. Let $(V,E,\Gamma,s,\psi)$ be a Kuranishi neighborhood of $(\Sigma, u)$ and let $(V^{1,2},E_1\oplus E_2,\Gamma^1 \times \Gamma^2,s^{1,2},\psi^{1,2})$ be a Kuranishi neighborhood of $(\Sigma^1,u_1)\times(\Sigma^2,u_2)$. Then, by construction $V^{1,2}=V^1 \times V^2$ where
\begin{align}
V^i=\big\{ \ (\Sigma^{i,+}(\eta_i),v_i) \ \big\vert \ v_i:(\Sigma^{i,+}(\eta),& \partial\Sigma^{i,+}(\eta))\rightarrow (M_i,L_i),\nonumber\\
&\bar\partial_{J_i} v_i\in E_i,v_i(w)\in N^i_w,w\in\nu(\Sigma^{i,+}) \ \big\}.\nonumber
\end{align}
Since
$$V=\big\{(\Sigma^+(\eta),v)\big\vert v:(\Sigma(\eta),\partial \Sigma(\eta))\rightarrow(M,L),\bar\partial_J v\in E_1\oplus E_2, v(w)\in N_w,w\in\nu(\Sigma^+)\big\},$$
there is a map $\varphi_\Pi: V \lto V^{1,2}$ defined as
$$\varphi_\Pi(\Sigma^+(\eta),v)=\big((\Sigma^{1,+}(\eta),p_1\circ v),(\Sigma^{2,+}(\eta),p_2\circ v)\big).$$
This map is locally modeled on the forgetful map $\mM_{l,1}\lto\mM_{l_1,1}\times\mM_{l_2,1}$ (with $l=l_1+l_2$). As above, this is a continuous map which is a submersion when restricted to each stratum.

This finishes the construction of the Kuranishi structures on $\mM_1(\beta_1)\times \mM_1(\beta_2)$ and $\mM_1(\beta_1\times \beta_2)$. From the construction it is obvious that they are compatible with the map $\Pi_1 \times \Pi_2$. Then, as before we construct Kuranishi structures on $\mM_{k+1}(\beta_1\times\beta_2)$ and  $\mM_p(\beta_1)\times \mM_q(\beta_2)$, for $k,p,q\geq0$, compatible with forgetting boundary marked points. Combining these two compatibilities we obtain Kuranishi structures compatible with all the maps $P_{I,J}$ and $Q_{i,I,J}$.
This completes the proof of Proposition \ref{compkuranishi}.

\subsection{Proof of Proposition \ref{b2=0}}\label{proof2}
With the exception of the statement about the orientations, this follows directly from the construction of Kuranishi structures in the proof of Proposition \ref{compkuranishi}.

In order to prove the statement on the orientations, we first recall some conventions on the orientations of the spaces $\mM_{k+1}(\beta)$, following \cite[Chapter 8]{FOOO}. The relative spin structure $\sigma$ on the Lagrangian $L$ determines a (stable) trivialization of $u^*TL$, restricted to the boundary of the disk for each map $u$. This gives a canonical orientation of the determinant line bundle $det D_u\bar\partial$, or equivalently of the determinant of Dolbeaut operator $det \bar\partial_u$ (see Section 8.1 of \cite{FOOO}). Which in turn determines an orientation of $\tilde \mM (\beta)$ the space of parametrized holomorphic disks. Then we define
$$\hat\mM_{k+1}(\beta)=\tilde\mM(\beta)\times (S^1)^{k+1}$$ where $(S^1)^{k+1}$ parametrizes the $(k+1)-$marked points on the boundary of the disk, ordered according to the usual orientation on $\partial D ^2= S^1$. The group $G= PSL_2(\mathbb{R})$ acts on the right of $\hat\mM_{k+1}(\beta)$ and $\mM_{k+1}(\beta)=\hat\mM_{k+1}(\beta)/G$. We define the orientation by
$$\mM_{k+1}(\beta)\times Lie(G)=\tilde\mM(\beta)\times (S^1)^{k+1}.$$

In our situation, if $\beta=\beta_1\times0$ we want to orient $det \bar\partial_u$ for $u=u_1\times u_2$, where $u_2$ is a constant map. Therefore we have $\ker\bar\partial_u=\ker \bar\partial_{u_1}\oplus TL_2$ and $coker\bar\partial_u=coker\bar\partial_{u_1}$. Therefore $det \bar\partial_u=det \bar\partial_{u_1}\otimes TL_2$ which by definition gives
$$\tilde \mM(\beta_1\times 0)=\tilde \mM(\beta_1)\times L_2.$$
Then following the definition of orientation on $\mM_{k+1}(\beta_1\times 0)$ we have
\begin{align}
&\tilde\mM(\beta_1\times 0)\times (S^1)^{k+1}=\mM_{k+1}(\beta_1\times 0)\times Lie (G)\nonumber\\
\Leftrightarrow &\tilde\mM(\beta_1)\times L_2\times (S^1)^{k+1}=\mM_{k+1}(\beta_1\times 0)\times Lie (G)\nonumber\\
\Leftrightarrow& (-1)^{n_2(k+1)}\tilde\mM(\beta_1)\times (S^1)^{k+1}\times L_2=\mM_{k+1}(\beta_1\times 0)\times Lie (G)\nonumber\\
\Leftrightarrow& (-1)^{n_2(k+1)}\mM_{k+1}(\beta_1)\times Lie(G)\times L_2=\mM_{k+1}(\beta_1\times 0)\times Lie (G)\nonumber\\
\Leftrightarrow& (-1)^{n_2k}\mM_{k+1}(\beta_1)\times L_2\times Lie(G)=\mM_{k+1}(\beta_1\times 0)\times Lie (G)\nonumber
\end{align}
where the last equality holds because $\dim G=3$. So we conclude
$$\mM_{k+1}(\beta_1\times0)=(-1)^{n_2k}\mM_{k+1}(\beta_1)\times L_2.$$ 
In the other case, $\beta=0\times\beta_2$, by the previous argument we have
$$det D_u\bar\partial = (-1)^{n_1 \dim coker D_{u_2}\bar\partial} det TL_1 \otimes det D_{u_2}\bar\partial.$$
Comparing with the definition of orientation on the product of Kuranishi spaces in Convention 8.2.1 in \cite{FOOO}, we conclude
$$\tilde \mM(0\times \beta_2)= L_1\times \tilde\mM(\beta_2).$$
The same argument then shows $\mM_{k+1}(0\times\beta_2)=L_1\times\mM_{k+2}(\beta_2).$
This completes the proof of Proposition \ref{b2=0}.

\subsection{Proof of Proposition \ref{compmultisection}}\label{proof3}
The strategy of the proof is similar to the proof of Proposition \ref{compkuranishi}. 

We first construct multisections on $\mM_{k+1}(\beta_l)$, for $l=1,2$ and $E(\beta_l)\leq E$ compatible with $\forg_i$. This was done by Fukaya in \cite{Fuk}, we simply highlight the main points. We first consider the situation of one Kuranishi chart $\mU_q$ in $\mM_1(\beta)$ constructed in Proposition \ref{compKur}. It is shown in \cite{Fuk}, that there is $\epsilon>0$, a vector space $W_q$ and a $W_q$-parametrized family of multisections $S_q:V_q\times W_q\lto \pi^*_q E_q$, such that $S_q$ is transversal, $\epsilon$- close to $s_q$ and $ev_{z_0}:S^{-1}_{q,i,j}(0)\lto L$ is a submersion. Then given a Kuranishi neighborhood $\mU_p$ in $\mM_{k+1}(\beta)$ with a map to $\mU_q$, we take $W_p=W_q$ and $S_p=S_q \circ \varphi$. 
The rest of the proof follows the usual argument on the area of the disks. The only difficulty is that the map $\varphi$ is not smooth, thus there is no guarantee that $S_p$ is smooth. However the map $\varphi$ is non-smooth only on directions transversal to each stratum in $\mM_{l,k}(\beta)$. Since the construction of multisections is by induction, we first define the multisection in the lower dimensional strata and then extend it to a neighborhood. Then one chooses the extension of the multisection so that it decays sufficiently fast in directions transversal to each stratum. Here is what we mean, if $y$ is a local coordinate perpendicular to the stratum and $T=1/y$, we require 
$$\Big\vert \frac{\partial^{k+l}S_q}{\partial T^k \partial x^l} \Big\vert< C e^{-cT},$$
for constants $C,c$, depending on $k,l$. As explained in \cite[page 778]{FOOO}, this condition is well defined since it is invariant under coordinate transformations of the Kuranishi structure. The map $\varphi$ is induced by the forgetful map $\forg_i$, thus it is locally either a submersion, when no components of the curve are contracted or, when one component is contracted, in the $T$ coordinate, it is given as $T=T_1+T_2$. Therefore $S_{p}=S_q \circ \varphi$ satisfies the same decay condition hence smooth. Moreover the multisection was already transversal  and $ev_{z_0}|_{S^{-1}_{q}(0)}$  was already submersive when restricted to the stratum, thus the decay condition in the direction normal to the stratum does not affect transversality. Therefore $S_p=S_q \circ \varphi$ is transversal and $ev_{z_0}|_{S^{-1}_{p}(0)}$ is submersive. Also note that, since $\varphi$ is smooth when restricted to each boundary strata, the map $\varphi\vert_{S^{-1}_{p}(0)}$ is, in particular, smooth on the top dimensional strata.

Next we construct continuous families of multisections on $\mM_1(\beta_1\times \beta_2)$ and $\mM_1(\beta_1)\times \mM_1(\beta_2)$ compatible with $\Pi_1 \times \Pi_2$, following the same inductive scheme. Recall from the proof of Proposition \ref{compkuranishi}, each Kuranishi neighborhood in $\mM_1(\beta_1)\times \mM_1(\beta_2)$ is the product of some Kuranishi neighborhoods on $\mM_1(\beta_1)$ and $\mM_1(\beta_2)$. From the above discussion, for $l=1,2$, we can take transversal multisections $S_{q_l}$ parametrized by $W_{q_l}$, such that $ev_0^l$ is a submersion when restricted to each stratum. Then define multisections $S_q=S_{q_1}\times S_{q_2}$ parametrized by $W_q=W_{q_1}\times W_{q_2}$ on $\mM_1(\beta_1)\times \mM_1(\beta_2)$. Additionally we impose the same decay conditions on directions transversal to each stratum. Then on $\mM_1(\beta_1\times \beta_2)$ we take $W_p=W_q$ and $S_p=S_q \circ \varphi_\Pi$. During the proof of Proposition \ref{compkuranishi} we saw that the map $\varphi_\Pi$ is locally 
modeled on 
the forgetful map $\mM_{l,1}\lto\mM_{l_1,1}\times\mM_{l_2,1}$, therefore it has the same local description as the map $\varphi$ discussed above. So we conclude that $S_p$ is smooth and transversal. Moreover, since $Ev_0= ev_0^1\times ev_0^2 \circ \varphi_\Pi$ and $\varphi_\Pi$ is a submersion when restricted to each stratum, we conclude that $Ev_0|_{S_p^{-1}(0)}$ is a submersion.

Repeating the above argument we construct multisections on $\mM_{k+1}(\beta_1\times\beta_2)$ and  $\mM_p(\beta_1)\times \mM_q(\beta_2)$, for $k,p,q\geq0$, compatible with forgetting boundary marked points. Combining these compatibilities, we obtain systems of multisections compatible with all the maps $P_{J,L}$ and $Q_{i+1,J,L}$.
This completes the proof of Proposition \ref{compmultisection}.

\section{From $\mF(L_1\times L_2)$ modulo $T^E$ to $\mF(L_1\times L_2)$}\label{secfukprodmodte}

In this section we will complete the proof of Theorem \ref{intcomsub}. So far we have shown (Proposition \ref{comsubalgmodte}) that for each $E>0$, $\mF(L_1)$ and $\mF(L_2)$ are commuting subalgebras modulo $T^{E}$ of $\mF(L_1\times L_2)$. We want to upgrade this to full-fledged \Ai-algebras. 

Let us introduce some notation. Consider $0<E_0<E_1$ and let $S_0$ and $S_1$ be two systems of multisections provided by Proposition \ref{compmultisection}, for energies $E_0$ and $E_1$ respectively. For $l=1,2$ denote by $\mF(L_1)^l$, $\mF(L_2)^l$ and $\mF(L_1\times L_2)^l$ the \Ai-algebras modulo $T^{E_l}$ determined by these systems of multisections. We know from Proposition \ref{comsubalgmodte}, that $\mF(L_1)^l$ and $\mF(L_2)^l$ are commuting subalgebras modulo $T^{E}$ of $\mF(L_1\times L_2)^l$. 

\begin{prop}\label{comsubisotopies}
 There exist three pseudoisotopies modulo $T^{E_0}$: $(\m^t_{k,\beta_1}, \mc^t_{k,\beta_1})$  between $\mF(L_1)^0$ and $\mF(L_1)^1$; $(\m^t_{k,\beta_2}, \mc^t_{k,\beta_2})$ between $\mF(L_2)^0$ an $\mF(L_2)^1$; and $(\m^t_{k,\beta_1\times \beta_2}, \mc^t_{k,\beta_1\times \beta_2})$ between $\mF(L_1\times L_2)^0$ an $\mF(L_1\times L_2)^1$. Using the notation of Proposition \ref{comsubalgmodte}, these pseudoisotopies satisfy the following relations
\begin{itemize}
\setlength{\parsep}{0pt}
\setlength{\itemsep}{0pt}
\item[{\bf(a)}] If $\beta_1,\beta_2\neq 0$ then $$\m^t_{k,\beta_1\times \beta_2}(\xi_1,\ldots,\xi_k)=0=\mc^t_{k,\beta_1\times \beta_2}(\xi_1,\ldots,\xi_k),$$
  $$\m^t_{k+1,\beta_1\times \beta_2}(\xi_1,\ldots,\xi_{i},\xi,\xi_{i},\ldots,\xi_k)=0=\mc^t_{k,\beta_1\times \beta_2}((\xi_1,\ldots,\xi_{i},\xi,\xi_{i},\ldots,\xi_k)).$$
  \item[{\bf(b)}] If $\xi_i=p_2^*(b_i)$, then
 $$\m_{k,\beta_1\times 0}(\xi_1,\ldots, \xi_k)=0=\mc^t_{k,\beta_1\times 0}(\xi_1,\ldots, \xi_k),$$
 and if $\xi_i=p_1^*(a_i)$, then
 $$\m_{k,0\times \beta_2}(\xi_1,\ldots, \xi_k)=0=\mc^t_{k,0\times \beta_2}(\xi_1,\ldots, \xi_k).$$
  \item[{\bf(c)}] 
  \begin{align}
         \m^t_{k,\beta_1 \times 0}(\iota(a_1),\ldots,\iota(a_k))&= \iota(\m^t_{k,\beta_1}(a_1,\ldots,a_k)),\nonumber\\
         \mc^t_{k,\beta_1 \times 0}(\iota(a_1),\ldots,\iota(a_k))&=(-1)^{n_2} \iota(\mc^t_{k,\beta_1}(a_1,\ldots,a_k)),\nonumber\\
         \m^t_{k+1,\beta_1\times 0}(\iota(a_1),\ldots,\iota(a_{i}),&K(a\otimes b),\iota(a_{i+1}),\ldots\iota(a_k))\nonumber\\
&=(-1)^{|b|\sum_{j> i}\|a_j\|}K(\m^t_{k+1,\beta_1}(a_1,\ldots,a,\ldots,a_k)\otimes b),\nonumber\\
         \mc^t_{k+1,\beta_1\times 0}(\iota(a_1),\ldots,\iota(a_{i}),&K(a\otimes b),\iota(a_{i+1}),\ldots\iota(a_k))\nonumber\\
&=(-1)^{n_2+|b|\sum_{j> i}\|a_j\|}K(\mc^t_{k+1,\beta_1}(a_1,\ldots,a,\ldots,a_k)\otimes b).\nonumber
        \end{align}
Finally the analogous statement for $\m^t_{k,0\times \beta_2}$ and $\mc^t_{0\times \beta_2}$ also holds.
 \end{itemize}
\end{prop}

Assuming this proposition we can complete the proof of Theorem \ref{intcomsub}.
\begin{cor}
 We can extend $\mF(L_1)^0$, $\mF(L_2)^0$ and $\mF(L_1\times L_2)^0$ to \Ai-algebras modulo $T^{E_1}$, denoted by $\mF(L_1)^{(1)}$, $\mF(L_2)^{(1)}$ and $\mF(L_1\times L_2)^{(1)}$, so that $\mF(L_1)^{(1)}$ and $\mF(L_2)^{(1)}$ are commuting subalgebras modulo $T^{E_1}$ of $\mF(L_1\times L_2)^{(1)}$
 
 Therefore we conclude that $\mF(L_1)$ and $\mF(L_2)$ are commuting subalgebras of $\mF(L_1\times L_2)$.
\end{cor}
\begin{proof}
From Proposition \ref{extendingmk} we know that we can use the pseudoisotopies to extend $\mF(L_1)^0$, $\mF(L_2)^0$ and $\mF(L_1\times L_2)^0$. Moreover we have a formula for the extensions given by \ref{pisotopyformula}. Combining this formula with the relations between the pseudoisotopies described in Proposition \ref{comsubisotopies}, we can easily see that the extensions to \Ai-algebras modulo $T^{E_1}$ satisfy the relations for commuting subalgebras in Definitions \ref{subalgebra} and \ref{comsubalg}.

Recall from Subsection \ref{modte}, that $\mF(L)$ is constructed by successively extending the \Ai-algebra modulo $T^{E_0}$, $\mF(L)^0$ to an \Ai-algebra modulo $T^{E_i}$ using the pseudoisotopies $(\m^t_{k,\beta}, \mc^t_{k,\beta})$. We have just shown that we can extend $\mF(L_1)^0$, $\mF(L_2)^0$ and $\mF(L_1\times L_2)^0$ so that $\mF(L_1)^{(i)}$ and $\mF(L_2)^{(i)}$ remain commuting subalgebras modulo $T^{E_i}$ of $\mF(L_1\times L_2)^{(i)}$. Thus in the limit, the \Ai-algebras $\mF(L_1)$ and $\mF(L_2)$ are commuting subalgebras of $\mF(L_1\times L_2)$.
\end{proof}

Now we are left with proving Proposition \ref{comsubisotopies}. The proof is similar to the one of Proposition \ref{comsubalgmodte}.

First note that there are natural maps 
$$\Pi_1:\mM^I_{k+1}(\beta_1\times\beta_2)\lto \mM^I_{k+1}(\beta_1) \ \ \text{and} \ \ \Pi_2:\mM^I_{k+1}(\beta_1\times\beta_2)\lto \mM^I_{k+1}(\beta_2)$$ 
given by $\Pi(t,(\Sigma,u))=(t,(\Sigma_i, u_i))$ where $\Sigma_i$ is obtained from $\Sigma$ by collapsing irreducible components that become unstable after forgetting the other component of $u$. Also note that there are natural maps $\forg_i:\mM^I_{k+1}(\beta) \lto \mM^I_{k}(\beta)$ that forget the i-th boundary marked point.  
Also note that $ev_t$ is a weak submersion, therefore the space $\mM^I_{k+1}(\beta_1)_{ev_t}\times_{ev_t}\mM^I_{k+1}(\beta_2)$ can be given a Kuranishi structure. Moreover the maps  $\Pi_l$ induce a map 
$$\Pi: \mM^I_{k+1}(\beta_1\times\beta_2) \lto \mM^I_{k+1}(\beta_1) \times_I \mM^I_{k+1}(\beta_2).$$
Using the notation from Definition \ref{projections}, we can define
$$P^I_{J,L}: \mM^I_{k+1}(\beta_1\times\beta_2) \lto \mM^I_{k-|L|+1}(\beta_1) \times_I \mM^I_{k-|J|+1}(\beta_2),$$
as the composition $(\forg_I \times_I \forg_J)\circ\Pi$ and 
$$Q^I_{i,I,J}: \mM^I_{k+2}(\beta_1\times\beta_2) \lto \mM^I_{k-|L|+2}(\beta_1) \times_I \mM^I_{k-|J|+2}(\beta_2),$$
as the composition $(\forg_J \times_I \forg_L)\circ\Pi$.

Now we have the analogue to Proposition \ref{compkuranishi}.
\begin{prop}
  There are Kuranishi structures on $\mM^I_{k}(\beta_1 \times \beta_2)$, $\mM^I_{k}(\beta_l)$, $(l=1,2)$ and $\mM^I_{p}(\beta_1)\times_I\mM^I_{q}(\beta_2)$ for $p,q,k \geq 1$ compatible with $P^I_{J,L}$, $Q^I_{i+1,J,L}$ and $\forg_j$. These Kuranishi structures respect the the decomposition of the boundary (\ref{boundaryparameter}) and the maps $Ev^I_0$, $ev^I_0$ are weak submersions. Moreover
  $$\mM^I_{k+1}(\beta_1\times 0)=(-1)^{n_2k}\mM^I_{k+1}(\beta_1)\times L_2,$$
$$\mM^I_{k+1}(0\times\beta_2)=(-1)^{n_1}L_1\times\mM^I_{k+1}(\beta_2).$$ 
\end{prop}
The proof is entirely analogous to the proof of Proposition \ref{compkuranishi}, so we do not repeat it. Next we have

\begin{prop}\label{compmultiparameter}
Fix $\epsilon, E> 0$ and let $S_0$ and $S_1$ be two compatible systems of multisections on $\mM_{k}(\beta_1 \times \beta_2)$, $\mM_{k}(\beta_1)$ for $l=1,2$ and $\mM_{p}(\beta_1)\times\mM_{q}(\beta_2)$ for $p,q,k \geq 1$, that satisfy the conditions of Proposition \ref{compmultisection}.
Then there exists a system of multisections $S$ on $\mM^I_{k}(\beta_1 \times \beta_2)$, $\mM^I_{k}(\beta_l)$, $(l=1,2)$ and $\mM^I_{p}(\beta_1)\times_I\mM^I_{q}(\beta_2)$, satisfying
\begin{itemize}
\setlength{\parsep}{0pt}
\setlength{\itemsep}{0pt}
\item[{\bf(a)}] $S\vert_{t=0}=S_0$ and $S\vert_{t=1}=S_1$;
\item[{\bf(b)}] they are transversal and $\epsilon-$close to the Kuranishi map;
\item[{\bf(c)}] they are compatible with the maps $P^I_{J,L}$, $Q^I_{i+1,J,L}$ and $\forg_j$;
\item[{\bf(d)}] $Ev^I_0\vert_{S^{-1}(0)}$ and $ev^I_0\vert_{S^{-1}(0)}$ are submersions;
\item[{\bf(e)}] they respect the boundary decomposition \eq{boundaryparameter}.
\end{itemize}
\end{prop}
Once again the proof of this proposition is analogous to the proof of Proposition \ref{compmultisection} so we omit it. We are now ready to prove Proposition \ref{comsubisotopies}.

\begin{proof}[Proof of Proposition \ref{comsubisotopies}]
We use the systems of multisections provided by Proposition \ref{compmultiparameter} to construct the maps $\m^I_{k,\beta_1\times \beta_2}$ and $\m^I_{k,\beta_l}$, for $l=1,2$, as in Proposition \ref{ainfinityI}. Then we use these to define the pseudoisotopies  $(\m^t_{k,\beta_1\times \beta_2}, \mc^t_{k,\beta_1\times \beta_2})$ and $(\m^t_{k,\beta_l}, \mc^t_{k,\beta_l})$ as in the formula (\ref{mkt}).

Following the proof of Proposition \ref{b1b2nonzero}, we can show, using compatibility with $P^I_{I,J}$ and $Q^I_{i,I,J}$ that
$$\m^I_{k,\beta_1\times \beta_2}(\xi_1,\ldots,\xi_k)=0,$$
$$\m^t_{k+1,\beta_1\times \beta_2}(\xi_1,\ldots,\xi_{i},\xi,\xi_{i},\ldots,\xi_k)=0,$$
when $\beta_1,\beta_2\neq0$. 
Next, following the proof of Proposition \ref{mkb10=0}, we show that, if $\xi_i=p_2^*(b_i)$, then
 $$\m^I_{k,\beta_1\times 0}(\xi_1,\ldots, \xi_k)=0,$$
 and if $\xi_i=p_1^*(a_i)$, then
 $$\m^I_{k,0\times \beta_2}(\xi_1,\ldots, \xi_k)=0.$$
Thus we have proved the first two conditions in Proposition \ref{comsubisotopies}. To prove the last one we follow the proof of Proposition \ref{comsubalgmodte}. There is only one difference in the sign of the formulas involving $\mc^t_{k,\beta}$. The reason for this is that $\mc^t_{k,\beta}$ has degree $k+1 \pmod 2$, unlike $\m^t_{k,\beta}$ which has degree $k+1 \pmod 2$, then by the definition of $\iota$ and $K$ we pick up the extra sign $(-1)^{n_2}$.
\end{proof}

\appendix\section{Smooth correspondences}
In this section, we first define fiber integration and collect several useful properties it satisfies. Then we review the construction and some of the properties of a smooth correspondence on a space with a good coordinate system and a continuous family of multisections, following \cite{Fuk} and \cite{FOOOt2}. See also \cite{FOOOnew}. Finally we prove Proposition \ref{kurvanishingpi*}.

\subsection{Fiber integration}
We start with the definition of \emph{fiber integration} on smooth manifolds.
\begin{defn}\label{fiberinteg}
Let $M$ and $N$ be smooth oriented manifolds, let $\pi:M^{n+k}\lto N^n$ be a submersion and $\alpha\in \Omega_c^m(M)$ a compactly supported $m$-form in $M$. We define the \emph{fiber integration} $\pi_*\a$ as the only $(m-k)$-form in $N$ that satisfies:
$$\int_M\a\wedge\pi^*\beta=\int_N\pi_*\a\wedge\beta,\ \ \forall\ \beta\in\Omega_c^{n-m+k}(N).$$
\end{defn} 
If  we take local coordinates $(t_1,\ldots,t_k,x_1,\ldots,x_n)$ in $M$ such that $(x_1,\ldots,x_n)$ are (oriented) coordinates in $N$ and $\pi(t_1,\ldots,t_k,x_1,\ldots,x_n)=(x_1,\ldots,x_n)$, $\pi_*\a$ can be described as follows. If $\a=f_I(x,t)dt_1\wedge\ldots \wedge dt_k\wedge dx_I+\ldots$ where all the other terms involve less then $k$ wedges of $dt_i$, we can check that
$$\pi_*\a=\left(\int f_I(x,t)dt_1\ldots dt_k\right)dx_I.$$
\begin{rmk}\label{rmkpi*}
In \cit{GuiSte} the definition of $\pi_*$ differs from ours by a sign. This happens because the order of $\a$ and $\beta$ in the definition of $\pi_*$ is reversed.
\end{rmk}

We will now state a few useful properties of $\pi_*$. 
\begin{prop}\label{proppi*}
 \begin{itemize}
\setlength{\parsep}{0pt}
\setlength{\itemsep}{0pt}
\item[{\bf(a)}] If $f$ and $g$ are submersions, then $(g\circ f)_*(\a)=g_*(f_*(\a)).$
\item[{\bf(b)}] If $\a\in\Omega^*_c(M)$ and $\gamma\in\Omega^*(N)$, then $\pi_*(\a\wedge\pi^*\gamma)=\pi_*\a\wedge\gamma$ and $\pi_*(\pi^*\gamma\wedge\a)=(-1)^{|\gamma|k}\gamma\wedge\pi_*\a$. 
\item[{\bf(c)}] Suppose we have the commutative square of maps
\begin{displaymath}
\xymatrix{ M\times_NN_1 \ar[r]^{\ \ \ \pi_2} \ar[d]_{\pi_1}& N_1\ar[d]^{g} \\
M\ar[r]^{\pi} & N.}
\end{displaymath} Then for $\a\in\Omega_c^*(M)$, $g^*\pi_*\a=(\pi_2)_*\pi_1^*\a$. 
\item[{\bf(d)}] Let $\pi_i: M_i^{n_i+k_i}\lto N_i^{n_i}$ be submersions and let $\rho_i\in\Omega^*(M_i)$ for $i=1,2$. Then 
$$(\pi_1\times \pi_2)_*(\rho_1\times\rho_2)=(-1)^{k_2(n_1+k_1+ |\rho_1|)}(\pi_1)_*\rho_1\times(\pi_2)_*\rho_2.$$
\end{itemize}
\end{prop}
\begin{proof}
The proofs of the first three statements are elementary and can be found in \cite[Section 10.1]{GuiSte}. For the last statement we compute
\begin{align}
\int_{M_1\times M_2}&\rho_1\times\rho_2\wedge(\pi_1\times \pi_2)^*(\beta_1\times\beta_2)\nonumber\\
&=\int_{M_1\times M_2}(-1)^{|\rho_2|(n_1+k_1-|\rho_1|)}\rho_1\wedge \pi_1^*\beta_1\wedge\rho_2\wedge \pi_2^*\beta_2\nonumber\\
&=(-1)^{|\rho_2|(n_1+k_1-|\rho_1|)}\int_{M_1}\rho_1\wedge \pi_1^*\beta_1\cdot\int_{M_2}\rho_2\wedge \pi_2^*\beta_2\nonumber\\
&=(-1)^{|\rho_2|(n_1+k_1-|\rho_1|)}\int_{N_1}(\pi_1)_*\rho_1\wedge\beta_1\cdot\int_{N_2}(\pi_2)_*\rho_2\wedge\beta_2\nonumber\\
&=(-1)^{|\rho_2|(n_1+k_1-|\rho_1|)}\int_{N_1\times N_2}(\pi_1)_*\rho_1\wedge\beta_1\wedge(\pi_2)_*\rho_2\wedge\beta_2\nonumber\\
&=(-1)^{|\rho_2|(n_1+k_1-|\rho_1|)}\int_{N_1\times N_2}(-1)^{(n+k_1-|\rho_1|)(\rho_2-k_2)}(\pi_1)_*\rho_1\wedge(\pi_2)_*\rho_2\wedge\beta_1\wedge\beta_2.\nonumber
\end{align}
So we conclude
$$(\pi_1\times \pi_2)_*(\rho_1\times\rho_2)=(-1)^{k_2(n_1+k_1-|\rho_1|)}(\pi_1)_*\rho_1\times(\pi_2)_*\rho_2$$.
\end{proof}
\begin{prop}\label{vanishingpi*}
 Consider maps $f:M\lto M_1$ and $g:M_1\lto N$ such that $\pi=g\circ f$ is a submersion. If $\dim M_1<\dim M$, then $\pi_*(f^*\a)=0$ for $\a\in\Omega^*_c(M_1)$.
\end{prop}
\begin{proof}
Locally we can find coordinates $(t_1,\ldots,t_k,x_1,\ldots,x_n)$ in $M$, $(s_1,\ldots,s_l,x_1,\ldots,x_n)$ in $M_1$ such that $(x_1,\ldots,x_n)$ are coordinates in $N$. Moreover $g(s_1,\ldots,s_l,x_1,\ldots,x_n)=(x_1,\ldots,x_n)$ and $f(t_1,\ldots,t_k,x_1,\ldots,x_n)=(f_1,\ldots,f_l,x_1,\ldots,x_n)$ for some local functions on $M$ $f_1,\ldots,f_l$. By assumption $l<k$, therefore in these coordinates $f^*\a$ does not have a summand involving $dt_1\wedge \ldots \wedge dt_k$. Then by definition $\pi_*(f^*\a)=0$.
\end{proof}

\begin{rmk}\label{vannotsmooth}
 The previous proposition has a generalization that will be useful later. Suppose that $f$ is a continuous map that is smooth only on an open, dense subset of $M$ and there are smooth maps $u$ and $v$ such that $u=v \circ f$. Then $\pi_*(u^*\beta)=0$. This happens since we can carry out the above argument in the open, dense subset where $f$ is smooth to conclude that $u^*\beta$ does not have a summand involving $dt_1\wedge \ldots \wedge dt_k$, on this subset. Continuity then implies this holds everywhere in $M$.
\end{rmk}

We need one additional property of fiber integration.
\begin{prop}\label{stokes}
 Let $M$ and $N$ be smooth oriented manifolds and let $\pi:M^{n+k}\lto N^n$ be a submersion. If $\partial N=\emptyset$, then
$$d\pi_*\a+(-1)^{k+1}\pi_*(d\a)=(-1)^{k+1}(\pi\vert_{\partial M})_*\a.$$
If $\partial N \neq \emptyset$, then we take the decomposition $\partial M=\partial ^+ M\cup\partial ^- M$ where $\partial ^-M=\pi^{-1}(\partial N)$. We have the following
$$d\pi_*\alpha+(-1)^{k+1}\pi_*d\a=(-1)^{k+1}(\pi\vert_{\partial^+M})_*\a.$$
\end{prop}
\begin{proof}
The proof of both statements is similar, we prove just the second one. We compute
\begin{align}
\int_Md\a\wedge\pi^*\beta&=\int_Md(\a\wedge\pi^*\beta)+(-1)^{|\a|+1}\a\wedge\pi^*d\beta\nonumber\\
&=\int_{\partial M}\a\wedge\pi^*\beta+(-1)^{|\a|+1}\int_M\a\wedge\pi^*d\beta\nonumber\\
&=\int_{\partial^- M}\a\wedge\pi^*\beta+\int_{\partial^+ M}\a\wedge\pi^*\beta+(-1)^{|\a|+1}\int_N\pi_*\a\wedge d\beta\nonumber\\
&=\int_{\partial^- M}\a\wedge\pi^*\beta+\int_N(\pi\vert_{\partial^+M})_*\a\wedge\beta+(-1)^{|\a|+1 +|\a|-k+1}\int_Nd(\pi_*\a)\wedge\beta+\nonumber\\
&\hspace{4cm}+(-1)^{|\a|+1+|\a|}\int_Nd(\pi_*\a\wedge d\beta)\nonumber\\
&=\int_{\partial^- M}\a\wedge\pi^*\beta+\int_N(\pi\vert_{\partial^+M})_*\a\wedge\beta+(-1)^k\int_Nd\pi_*\a\wedge\beta-\int_{\partial N}\pi_*\a\wedge\beta\nonumber\\
&=\int_N\left[(\pi\vert_{\partial^+M})_*\a+(-1)^k\pi_*\a\right]\wedge\beta,\nonumber
\end{align}since $\int_{\partial^-M}\a\wedge\pi^*\beta=\int_{\partial N}\pi_*\a\wedge\beta$.
\end{proof}

\subsection{Smooth correspondences}\label{smoothcorr}
In this subsection we review the construction of a smooth correspondence, mostly following \cite{FOOOt2}. The only difference is that we require that all the auxiliary spaces parameterizing multisections are even dimensional. This is always possible and simplifies a lot of sign considerations. 

Let $X$ be a space with an oriented Kuranishi structure and a good coordinate system $\{(V_\a, E_\a,\Gamma_\a,\psi_\a,s_\a)\}_{\a\in I}$.
For each $\a$, let $S^l(E_\a)$ be the quotient of the vector bundle $\bigoplus_{j=1}^l E_\a$ by the action of the symmetric group in $l$ elements. There are natural maps
$$t_m:S^l(E_\a)\lto S^{lm}(E_\alpha), \ t_m[a_1\ldots a_l]=[\underbrace{a_1\ldots a_1}_{m},\ldots,\underbrace{a_l\ldots a_l}_{m}].$$
\begin{defn}\label{multi} A \emph{multisection} of $(E_\alpha\lto V_\alpha,\Gamma_\a)$ consists of an open cover $V_\a=\bigcup_{i\in A} U_i$ and a of $S^{l_i}(E_\alpha)\vert_{U_i}$, $s_i$,  satisfying: 
 \begin{itemize}
\setlength{\parsep}{0pt}
\setlength{\itemsep}{0pt}
\item[{\bf(a)}] $U_i$ is $\Gamma_\a$ -invariant and $s_i$ is $\Gamma_\a$ -equivariant,
\item[{\bf(b)}] if $x\in U_i\cap U_j$, then $t_{l_j}(s_i(x))=t_{l_i}(s_j(x))\in S^{l_il_j}(E_\a)$,
\item[{\bf(c)}] for each $x$, there is a smooth section $\tilde s_i=(\tilde s_{i,1}\times\ldots\times \tilde s_{i,l_i})$ of $E^{\oplus l_i}$ on a neighborhood of $x$ that represents $s_i$, i.e.
$$[\tilde s_i(y)]=[(\tilde s_{i,1}(y),\ldots,\tilde s_{i,l_i}(y))]=s_i(y)$$
(we call $\tilde s_{i,j}$ a branch of $s_i$).
\end{itemize} We identify two multisections $(U_i,s_i)_i, (U'_j,s'_j)_j$,  if $t_{l'_j}(s_i(x))=t_{l_i}(s'_j(x))\in S^{l_il'_j}(E_\a)$, for $x\in U_i\cap U'_j$.
\end{defn}

Next, we review the notion of \emph{continuous family of multisections on $X$}. Let $W_\a$ be a smooth oriented manifold of even dimension and let $\theta_\a$ be a compactly supported volume form such that $\int_{W_\a}\theta_\a=1$. Consider the pull-back bundle $\pi_\a^*E_\a\lto W_\a \times V_\a$, under the projection $\pi_\a:W_\a \times V_\a\lto V_\a$. We extend the action of $\Gamma_\a$ to $\pi_\a^*E_\a\lto W_\a \times V_\a$ by making it act trivially on $W_\a$.
\begin{defn}\label{paramulti}
A \emph{$W_\a$ -parametrized family $S_\a$ of multisections} is a multisection of $\pi_\a^*E_\a$. We say $S_\a$ is $\epsilon$-small if (after fixing a metric on $E_\a$) for each branch $S_{\a,i,j}$ of $S_\a$ we have $$\vert S_{\a,i,j}(w,\cdot)-s_\a(\cdot)\vert_{C^0}<\epsilon,\ \textrm{for each }w\in W_\a.$$
Finally, we say $S_\a$ is \emph{transversal} if each branch $S_{\a,i,j}$ is transversal to the zero-section of $\pi_\a^*E_\a$.
\end{defn}
Let $f_\a:V_\a\lto M$ be a $\Gamma_\a$ -invariant smooth map and assume $V_\a$ has a transversal multisection $S_\a$. We say $f_\a\vert_{S^{-1}_\a(0)}$ is a \emph{submersion} if for each branch  $S_{\a,i,j}$ of $S_\a$ the restriction $$f_\a\circ\pi_\a:S^{-1}_{\a,i,j}(0)\lto M$$ is a submersion. If $V_\a$ has corners, we require that the restriction of the above map to any boundary stratum is a submersion. With these definitions we have the following:
\begin{lem}[{\cite[Lemma 12.4]{FOOOt2}}]\label{loctransvsect}
If $f_\a:V_\a\lto M$ is a submersion, then there exists $W_\a$ and a $W_\a$ -parametrized family $S_\a$ of transversal multisections, $\epsilon$-small, such that $f_\a\vert_{S^{-1}_\a(0)}$ is a submersion.

If $S_\a$, satisfying these conclusions, is already defined on the neighborhood of a $\Gamma_\a$ -invariant compact subset $K_\a\subseteq V_\a$, then it can be extended to $V_\a$.
\end{lem}
The multisections on different Kuranishi neighborhoods are required to satisfy several compatibility relations. We omit them, but they can be found in \cite{FOOOt2}. 

\begin{defn}
 Let $Y$ be a smooth manifold. A continuous map $f:X\lto Y$ is said to be \emph{smooth strongly continuous} if there is a family of $\Gamma_\a$ -invariant smooth maps $f_\a:V_\a\lto Y$ such that $f_{\a}\circ\phi_{\a \beta}=f_{\beta}$, inducing $f$ (this means $f\circ\psi_\a\vert_{s^{-1}_\a/\Gamma_\a}$). We say that $f$ is \emph{weakly submersive} if, in addition, for each $\a$, the restriction of $f_\a$ to each boundary stratum is a submersion.
\end{defn}

Using these compatibility conditions and induction on $\a$ with respect to the order $\leq$, we have the following:
\begin{lem}[{\cite[Lemma 12.9]{FOOOt2}}]\label{globaltransvsect}
Let X be a Kuranishi space with a good coordinate system $\{\mU_\a\}_{\a\in I}$ and $f=\{f_\a\}_{\a\in I}:X\lto M$ a weak submersion. Then there is a continuous family of multisections $(W_\a,S_\a)$, transversal and $\epsilon$-small. Moreover, $f_\a\vert_{S^{-1}_\a(0)}$ is a submersion. A relative version of this result also holds.
\end{lem}

We are finally ready to define smooth correspondence. Let $X$ be a Kuranishi space with a good coordinate system $\{\mU_\a\}_{\a\in I}$ and let $M$ and $N$ be (oriented) smooth manifolds. Also consider smooth strongly continuous maps $f: X \lto M$ and $g: X \lto N$. We assume $g$ is weakly submersive and fix  a continuous family of multisections $(W_\a,S_\a)$ such that $g_\a\vert_{S^{-1}_\a(0)}$ is a submersion. We define
$$\Corr(f,X^S,g): \Omega^*(M) \lto \Omega^{*- \vdim X}(N).$$

We need an auxiliary partition of unity on $X$ subordinated to $\{\mU_\a\}_{\a\in I}$. This consists of a family $\mathcal{X}_\a:V_\a\lto\mathbb{R}$ of compactly supported, $\Gamma_\a$ -invariant smooth functions, satisfying a compatibility condition (see \cite[Definition 12.10]{FOOOt2}, for the precise definition).

Given $\xi\in\Omega^*(M)$, define $\xi_\a=\mathcal{X}_\alpha\cdot(f_\alpha\circ\pi_\alpha)^*\xi\in\Omega^*(W_\a\times V_\a)$. We will first define the $\mU_\a$ -component of $\Corr(f,X^S,g)(\xi)$. Recall from Definition \ref{multi} of multisection, we have $V_\a=\bigcup_{i\in A} U_{\a,i}$ and $S_{\a,i}$ a multisection of $\pi_\a^*E_\a$ on $U_{\a,i}$. Pick a partition of unity $\{\tau_i\}_{i\in A}$ subordinated to this covering. Without loss of generality we can assume the $\tau_i$ are $\Gamma_\a$ -invariant. Pick also a lifting $\widetilde S_{\a,i}=(\widetilde S_{\a,i,1},\ldots, \widetilde S_{\a,i,l_i})$ of $S_{\a,i}$. By definition $\widetilde S_{\a,i,1}^{-1}(0)$ is a submanifold of $W_\a\times V_\a$ and $g_\a\circ\pi_\a\vert_{S_{\a,i,j}^{-1}(0)}: S_{\a,i,j}^{-1}(0)\lto N$ is a submersion. We now define:
\begin{align}\label{localcorr}
(\mU_\a, S_\a,g_\a)_*(\xi_\a):=\frac{1}{\vert\Gamma_\a\vert}\sum_{i\in A}\sum_{j=1}^{l_i}\frac{1}{l_i}\left(f_\a\circ\pi\vert_{\widetilde S_{\a,i,j}^{-1}(0)}\right)_*(\tau_i\xi_\a\wedge\theta_\a).
\end{align}Recall that $\theta_\a$ is a compactly supported volume form in $W_\a$ of total volume $1$.

One can check that this is independent of the choice of the representative $(U_i, S_{\a,i})_i$, the lifts $\widetilde S_{\a,i}$ and the partition of unity (see \cite[Lemma 12.6]{FOOOt2}, for the proof). Finally we give the following
\begin{defn}\label{corrXS}
Let $S=(W_\a,S_\a)$ be an $\epsilon$-small, transversal continuous family of multisections on $X$. We define
$$\Corr(f,X^S,g)(\xi)=\sum_{\a\in I}(\mU_\a, S_\a,g_\a)_*(\xi_\a).$$
\end{defn}
Again one can check this definition is independent of the choice of partition of unity (see \cite[Remark 12.12]{FOOOt2}). However, it \emph{depends} on the choice of multisections $(W_\a,S_\a)$ and volume form $\theta_\a$, and the notation $X^S$ serves as a reminder of this.

We next prove two properties of smooth correspondences. The first one is a generalization of Proposition \ref{stokes} to Kuranishi spaces and the second gives a formula for composition of correspondences. With exception of signs, these can be found in \cite{Fuk} and \cite{FOOOt2}.

\begin{prop}\label{kurstokes} Let $X$,$f$,$g$, $M$ and $N$ be as before and denote by $k=v\dim X-\dim M$. If $\partial N = \emptyset$, then 
$$d\Corr(f,X^S,g)(\xi)+(-1)^{k+1}\Corr(f,X^S,g)(d\xi)=(-1)^{k+1}\Corr(f\vert_{\partial X},\partial X^S,g\vert_{\partial X})(\xi).$$
If $\partial N \neq \emptyset$, then we take the decomposition $\partial X=\partial ^+ X\cup\partial ^- X$ and we have
$$d\Corr(f,X^S,g)(\xi)+(-1)^{k+1}\Corr(f,X^S,g)(d\xi)=(-1)^{k+1}\Corr(f\vert_{\partial^+ X},\partial X^S,g\vert_{\partial^+ X})(\xi).$$
\end{prop}
\begin{proof} By definition of correspondence, it is enough to consider the case when there is only one Kuranishi neighborhood. Then it is enough to consider each branch of the multisection separately. In this case the proposition follows from Proposition \ref{stokes} applied to the form $\mathcal{X}_i\xi_\a\wedge\theta_\a.$
\end{proof}

Next we discuss composition of correspondences. Consider spaces with Kuranishi structures and continuous families of multisections $(X^1,S^1)$ and  $(X^2, S^2)$. Let  $M$, $M_1$, $M_2$, $L$ and  $N$ be closed manifolds and suppose we have weak submersions $\pi_1:X^1\lto M$, $\pi_2:X^2\lto M$ and $g:X^2 \lto N$ and smooth strongly continuous $\varphi_1:X^2\lto M_1$, $\varphi_2:X^2\lto M_2$ and $f: X^1\lto L$.

Then $X^1\times_MX^2$ has a Kuranishi structure and there are natural projections $p_i:X^1\times_M X^2\lto X^i$. It also has an induced continuous family of multisections, namely $(W^1_\a\times W^2_\beta,S^1_\a\times_M S^2_\beta)$ with volume form $\theta^1_\a\wedge\theta_\beta^2$.  Also the map induced by $g$, $g \circ p_2:X^1\times_M X^2\lto N$  is also a weak submersion, see \cite[Section 2]{Fuk}. We then have the following
\begin{prop}\label{composition} Given $\xi_1\in\Omega^*(M_1)$, $\xi_2\in\Omega^*(L)$, $\xi_3\in \Omega^*(M_2)$, we have
\begin{align}
&\Corr(\varphi_1 p_2\times f p_1\times \varphi_2 p_2,X^1\times_MX^2,\ g p_2)(\xi_1\times\xi_2\times\xi_3)\nonumber\\
&=(-1)^{k |\xi_1|}\Corr(\varphi_1\times \pi_2\times \varphi_2,X^2,g)\left(\xi_1\times \Corr(f,X^1,\pi_1)(\xi_2)\times \xi_3\right),\nonumber\end{align}
where $k=\vdim X^1 - \dim M$.
\end{prop}
\begin{proof} As is the previous proposition we only need to prove this equality on a single Kuranishi neighborhood, in which case it reduces to the same statement for manifolds, applied to the forms $\mathcal{X}^1_i \mathcal{X}^2_j (\xi_1\times\xi_2\times\xi_3)\wedge\theta^1_\a\wedge\theta_\beta^2$. Note that since the $\theta$'s have even degree the signs are not altered. We compute
\begin{align}
&\Corr(\varphi_1 p_2\times f p_1\times \varphi_2 p_2,X^1\times_MX^2,\ g p_2)(\xi_1\times\xi_2\times\xi_3)\nonumber\\
&=g_*\Big((p_2)_*\big(p_2^*(\varphi_1^*\xi_1)\wedge p_1^*(f^*\xi_2)\wedge p_2^*(\varphi_2^*\xi_3)\big)\Big)\nonumber\\
&=(-1)^{|\xi_1||\xi_2|}g_*\Big((p_2)_*\big(p_1^*(f^*\xi_2)\wedge p_2^*(\varphi_1^*\xi_1\wedge \varphi_2^*\xi_3)\big)\Big)\nonumber\\
&=(-1)^{|\xi_1||\xi_2|}g_*\Big((p_2)_*\big(p_1^*(f^*\xi_2)\big)\wedge \varphi_1^*\xi_1\wedge \varphi_2^*\xi_3\Big)\nonumber\\
&=(-1)^{|\xi_1||\xi_2|}g_*\Big((\pi_2)^*\big((\pi_1)_*f^*\xi_2\big)\wedge \varphi_1^*\xi_1\wedge \varphi_2^*\xi_3\Big)\nonumber\\
&=(-1)^{|\xi_1||\xi_2|+ |\xi_1|(|\xi_2|+k)}g_*\Big(\varphi_1^*\xi_1\wedge \pi_2^*\Corr(f,X^1,\pi_1)(\xi_2)\wedge  \varphi_2^*\xi_3\Big)\nonumber\\
&=(-1)^{|\xi_1|k}\Corr(\varphi_1\times \pi_2\times \varphi_2,X^2,g)\left(\xi_1\times \Corr(f,X^1,\pi_1)(\xi_2)\times\xi_3\right),\nonumber
\end{align}
here the first, third and forth equalities follow from, respectively, Proposition \ref{proppi*}(a),(b) and (c).
\end{proof}

Finally we prove Proposition \ref{kurvanishingpi*}.

\begin{proof}[Proof of Proposition \ref{kurvanishingpi*}]
It is enough to check that the contribution on each Kuranishi neighborhood vanishes, that is
\begin{align}
\left(f_\alpha\vert_{S^{-1}_{\alpha,i,j}(0)}\right)_*\left(\tilde{\tau}_i\ \xi_\alpha\wedge\theta_\alpha\right)=0.\label{factorpush}
\end{align}
Here $\tilde{\tau}_i$ is a partition of unity subordinated to $\varphi^{-1}_{\alpha\beta}(\mathcal{U}_{\alpha,i})$ for a cover $\bigcup_i\mathcal{U}_{\beta,i}=\mathcal{V}_\beta$  and the form $\xi_\alpha=\chi_\alpha(f'_\alpha)^*(\xi)\in \Omega^*(W_\alpha\times V_\alpha)$ for $\chi_\alpha$ a partition of unity subordinated to  $\{\mathcal{U}_\alpha\}_{\alpha\in I}$.

The compatibility of the multisections implies that for each branch of the multisections $\varphi_{\alpha\beta}$ there is an induced map
$$\textrm{id}\times\varphi_{\alpha\beta}\vert_{S^{-1}_{\alpha,i,j}(0)}:S^{-1}_{\alpha,i,j}(0)\lto S^{-1}_{\beta,i,j}(0),$$
which by assumption is smooth on an open, dense subset. Note that $\dim S^{-1}_{\alpha,i,j}(0)= \vdim X + \dim W_\alpha$ and $\dim S^{-1}_{\beta,i,j}(0)= \vdim Y + \dim W_\a$.
Therefore we are in the situation of Remark \ref{vannotsmooth} since $f'_\alpha\vert_{S^{-1}_{\alpha,i,j}(0)}=g'_\alpha\circ\varphi_{\alpha\beta}\vert_{S^{-1}_{\alpha,i,j}(0)}$ and $f_\alpha\vert_{S^{-1}_{\alpha,i,j}(0)}=g'\circ\varphi_{\alpha\beta}\vert_{S^{-1}_{\alpha,i,j}(0)}$. Therefore the fiber integration in (\ref{factorpush}) vanishes.
\end{proof}

\bibliographystyle{abbrv}
\bibliography{biblio_tensor}

\begin{thebibliography}{10}

\bibitem{AkaJoy}
M.~Akaho and D.~Joyce.
\newblock Immersed {L}agrangian {F}loer theory.
\newblock {\em J. Differential Geom.}, 86(3):381--500, 2010.

\bibitem{Amo2}
L.~{Amorim}.
\newblock {Tensor product of filtered {$A_\infty$}-algebras}.
\newblock {\em J. Pure Appl. Algebra}, 220(12):3984--4016, 2016.

\bibitem{Flo}
A.~Floer.
\newblock Morse theory for {L}agrangian intersections.
\newblock {\em J. Differential Geom.}, 28(3):513--547, 1988.

\bibitem{Fuk}
K.~Fukaya.
\newblock Cyclic symmetry and adic convergence in {L}agrangian {F}loer theory.
\newblock {\em Kyoto J. Math.}, 50(3):521--590, 2010.

\bibitem{FOOO}
K.~Fukaya, Y.-G. Oh, H.~Ohta, and K.~Ono.
\newblock {\em Lagrangian intersection {F}loer theory: anomaly and obstruction.
  {P}arts {I} and {II}}, volume~46 of {\em AMS/IP Studies in Advanced
  Mathematics}.
\newblock American Mathematical Society, Providence, RI, 2009.

\bibitem{FOOOt1}
K.~Fukaya, Y.-G. Oh, H.~Ohta, and K.~Ono.
\newblock Lagrangian {F}loer theory on compact toric manifolds. {I}.
\newblock {\em Duke Math. J.}, 151(1):23--174, 2010.

\bibitem{FOOOt2}
K.~Fukaya, Y.-G. Oh, H.~Ohta, and K.~Ono.
\newblock Lagrangian {F}loer theory on compact toric manifolds {II}: bulk
  deformations.
\newblock {\em Selecta Math. (N.S.)}, 17(3):609--711, 2011.

\bibitem{FOOOtech}
K.~Fukaya, Y.-G. Oh, H.~Ohta, and K.~Ono.
\newblock Technical details on kuranishi structure and virtual fundamental
  chain.
\newblock {\em arXiv:1209.4410}, 2012.

\bibitem{FOOOnew}
K.~Fukaya, Y.-G. Oh, H.~Ohta, and K.~Ono.
\newblock Kuranishi structure, pseudo-holomorphic curve, and virtual
  fundamental chain: Part 1.
\newblock {\em arXiv:1503.07631}, 2015.

\bibitem{GuiSte}
V.~W. Guillemin and S.~Sternberg.
\newblock {\em Supersymmetry and equivariant de {R}ham theory}.
\newblock Mathematics Past and Present. Springer-Verlag, Berlin, 1999.
\newblock With an appendix containing two reprints by Henri Cartan [ MR0042426
  (13,107e); MR0042427 (13,107f)].

\bibitem{Joy}
D.~Joyce.
\newblock On manifolds with corners.
\newblock In {\em Advances in geometric analysis}, volume~21 of {\em Adv. Lect.
  Math. (ALM)}, pages 225--258. Int. Press, Somerville, MA, 2012.

\bibitem{Seisub}
P.~Seidel.
\newblock {$A_\infty$}-subalgebras and natural transformations.
\newblock {\em Homology, Homotopy Appl.}, 10(2):83--114, 2008.

\end{thebibliography}
\medskip

\noindent \textbf{Address:}
Boston University, Department of Mathematics and Statistics, 111 Cummington Mall, Boston MA, USA. 

\noindent \textbf{E-mail:} {\tt lamorim@bu.edu}

\end{document}